\documentclass[12pt,a4paper]{article}

\usepackage[english]{babel}
\usepackage[utf8x]{inputenc}
\usepackage[T1]{fontenc}

\usepackage[a4paper,top=2.5cm,bottom=2.5cm,left=2cm,right=2cm,marginparwidth=1.75cm]{geometry}

\usepackage{amsmath,amsthm,amsfonts}
\usepackage{graphicx}
\usepackage[colorinlistoftodos]{todonotes}
\usepackage[colorlinks=true, allcolors=blue]{hyperref}
\usepackage{scalerel}
\usepackage{stmaryrd}
\usepackage{relsize}
\usepackage{enumitem}

\newcommand\blfootnote[1]{%
  \begingroup
  \renewcommand\thefootnote{}\footnote{#1}%
  \addtocounter{footnote}{-1}%
  \endgroup
}

\newtheorem{theorem}{Theorem}[section]
\newtheorem{lemma}[theorem]{Lemma} 
\newtheorem{proposition}[theorem]{Proposition}
\newtheorem{corollary}[theorem]{Corollary}
\theoremstyle{definition}
\newtheorem{definition}[theorem]{Definition} 
\newtheorem{example}[theorem]{Example}

\theoremstyle{remark}
\newtheorem{remark}{Remark}

\newcommand{\C}{\mathbb{C}}
\newcommand{\R}{\mathbb{R}}
\newcommand{\N}{\mathbb{N}}
\newcommand{\Z}{\mathbb{Z}}

\def\Bsp{{\boldsymbol B}} 
\def\Xsp{{\boldsymbol X}}

\newcommand{\SO}{\textnormal{\textbf{S}}_{0}}
\newcommand{\SOprime}{\textnormal{\textbf{S}}'_{0}}
\newcommand{\SOdoubleprime}{\textnormal{\textbf{S}}''_{0}}

\newcommand{\ws}{weak$^{*}$}
\newcommand{\ghat}{\widehat{G}}
\newcommand{\TFelement}{\nu} 
\newcommand{\netparam}{\alpha} 
\newcommand{\Lp}[1]{\mathbf{L}^{#1}}
\newcommand{\Bil}{\mathbf{Bil}}
\newcommand{\Lin}{\mathbf{Lin}}

\newcommand{\Misp}{\mathbf{M}^1}

\newcommand{\bigsubseteq}{\mathlarger{\mathlarger{\subseteq}}}
\newcommand{\rotsubseteq}{\mathrel{{\rotatebox[origin=c]{270}{$\bigsubseteq$}}}}
\newcommand{\rotasubseteq}{\mathrel{{\rotatebox[origin=c]{45}{$\bigsubseteq$}}}}
 
\usepackage{tikz}
\usetikzlibrary{arrows,matrix,positioning}

\title{The inner kernel theorem for a certain Segal algebra}
\author{Hans G.\ Feichtinger\thanks{University of Vienna, Faculty of Mathematics, Vienna, Austria, ORCID: 0000-0002-9927-0742,  \mbox{E-mail: \protect\url{hans.feichtinger@univie.ac.at}}} \ and Mads S.\ Jakobsen\thanks{corresponding author, Center for Data and Computing in Natural Science CDCS, Deutsches Elektronen-Synchrotron DESY, Germany, ORCID: 0000-0003-0421-5343, \mbox{E-mail: \protect\url{mads.jakobsen@desy.de}}}}

\begin{document} 

\maketitle
\begin{abstract}
The Segal algebra $\SO(G)$ is well defined for arbitrary 
locally compact Abelian Hausdorff (LCA) groups $G$. It is a Banach space that exhibits a kernel theorem similar to the well-known Schwartz kernel theorem. 
Specifically, we call this 
characterization of the continuous linear operators from $\SO(G_1)$
to $\SOprime(G_2)$ by generalized functions in $\SOprime(G_1 \times G_2)$ the ``outer kernel theorem''. The main subject of this paper is to formulate what we call the ``inner kernel theorem''. This is the characterization of those linear operators 
that have kernels in $\SO(G_1 \times G_2)$. Such operators are regularizing -- in the sense that they map $\SOprime(G_1)$ 
into $\SO(G_2)$ in a $w^{*}$ to norm
continuous manner. A detailed functional
analytic treatment of these operators is given and applied to the case of
general LCA groups. This is done without the use of
Wilson bases, which have previously been employed for the case of elementary LCA 
groups. We apply our approach to describe natural laws of composition for operators that
imitate those of linear mappings via matrix multiplications. Furthermore, we detail how these operators approximate general operators (in a weak form). 
As a concrete example, we derive the widespread statement of engineers and physicists that pure frequencies ``integrate'' to a Dirac delta distribution in a mathematically justifiable way.
\end{abstract} 


\section{Introduction}

\blfootnote{Keywords: Feichtinger Algebra; Test functions; Generalized Functions; Kernel Theorem; Nuclear Operators; Time-frequency analysis}
\blfootnote{AMS Classification 2020: 41A65, 43A15, 43A25, 44A05, 46A11, 46F05, 46F10, 47B10, 47B34}

\noindent The focus of this paper is the kernel theorem associated with the Segal algebra $\SO(G)$. This space of functions was introduced by the first named author in \cite{fe81-2}.  
Given a locally compact Abelian Hausdorff (LCA) group $G$ we write $\ghat$ for its dual group, and for each $\omega\in \ghat$ we denote by 
$E_{\omega} f(t) = \omega(t) f(t), t\in G$  the \emph{modulation-} or \emph{frequency-shift} operator. We  define the set $\SO(G)$ via convolution ``$*$'' and the usual norm in $\Lp{1}$: 
\begin{equation}
\label{defSO1}
\SO(G) = \Big\{ f \in \Lp{1}(G) \, : \,  \int_{\ghat} \Vert E_{\omega} f * f\Vert_{1} \, d\omega < \infty \ \Big\}. 
\end{equation} 
Any non-zero function $g \in \SO(G)$ (also called \emph{window}) defines a norm on $\SO(G)$ via
\begin{equation}
\label{defSO2}
\Vert f \Vert_{\SO,g} = 
\Vert f \Vert_{\SO(G),g} := \int_{\ghat} \Vert E_{\omega} f * g\Vert_{1} \, d\omega,
\end{equation}
that turns $\SO$ into a Banach space.  For any non-zero funcions $g_{1}$ and $g_{2}$ in $\SO(G)$ these norms are pairwise equivalent and we therefore allow ourselves to simply write $\Vert  \cdot  \Vert_{\SO}$ without specifying the function $g$. 
The space $\SO(G)$ is a Fourier invariant Banach algebra under convolution and pointwise multiplication. 
The continuous linear functionals on this space form a space of \emph{generalized functions}.
Altogether they comprise the dual space 
 $\SOprime(G)$, which is a Banach space itself. The action $\sigma\in\SOprime(G)$ on a function $f\in \SO(G)$ is described by the bilinear form
\begin{equation}
\label{bilinSO2}
 ( \, \cdot \, , \, \cdot \, )_{\SO,\SOprime(G)}  \,  : \, \SO(G) \times \SOprime(G) \to \C, 
 \quad ( f, \sigma)_{\SO,\SOprime(G)} = \sigma(f).
\end{equation}  

Throughout the paper $\Bil(X\times Y,Z)$ is the space of bilinear and norm continuous operators from the normed space $X\times Y$ into the normed space $Z$. Similarly, $\Lin(X,Y)$ is the space of linear and norm continuous operators from $X$ into $Y$. The spaces $\Bil$ and $\Lin$ are endowed with their natural operator norms.

Using these spaces 
we can formulate the following result.
\begin{theorem}[Outer kernel theorem for $\SO$] \label{th:outer-kernel-theorem}
 For any two LCA groups $G_{1}$ and $G_{2}$ the following four Banach spaces 
 are naturally isomorphic:
\[ \SOprime(G_{1}\times G_{2}), \ \Bil(\SO(G_{1})\times \SO(G_{2}), \C) , \ \Lin(\SO(G_{1}),\SOprime(G_{2})) \ \text{and} \ \Lin(\SO(G_{2}),\SOprime(G_{1})). \] 
In particular, given any of the four $K \in \SOprime(G_{1}\times G_{2})$,
\[ A\in\Bil(\SO(G_{1})\times \SO(G_{2}), \C), \ T\in\Lin(\SO(G_{1}),\SOprime(G_{2})) \ \text{or} \ S\in\Lin(\SO(G_{2}),\SOprime(G_{1}))\] 
the others are uniquely determined by the following identity, valid for all $f^{(i)}\in\SO(G_{i})$, $i=1,2$: 
\[ (f^{(1)}\otimes f^{(2)}, K)_{\SO,\SOprime(G_{1}\times G_{2})} \ = \ A(f^{(1)},f^{(2)}) \ = \  ( f^{(2)},Tf^{(1)})_{\SO,\SOprime(G_{2})} \ = \ ( f^{(1)},S f^{(2)})_{\SO,\SOprime(G_{1})}. \]
\end{theorem} 

The unique generalized function $K \in \SOprime(G_{1}\times G_{2})$  associated with $A, T$ or 
$S$ is called the \emph{kernel} of $A,T$ or $S$, respectively and we write $\kappa(A) = \kappa(T) = \kappa(S) = K$. The outer kernel theorem for $\SO$ was first announced in \cite{fe80}. Its proof can be found in, for example,  \cite{fegr92-1,feko98,ja19}. 

\vspace{2mm}  \noindent 
This paper considers the following question: 
\begin{enumerate}
\item[]\textit{Is there an analogue of Theorem \ref{th:outer-kernel-theorem} concerning operators that can be naturally identified with the functions in $\SO(G_{1}\times G_{2})$ (rather than its dual space $\SOprime(G_{1}\times G_{2})$)?} 
\end{enumerate}
This question has been considered and answered before in \cite{cofelu08} and \cite{feko98}, however not in the generality considered here (cf.\ Remark \ref{rem:2601c} following Theorem \ref{th:new-inner-kernel-theorem} below).
As is well known (and as we will explain in detail in Section \ref{sec:preliminaries}) there is a natural isomorphic copy of the Banach space of functions $\SO(G)$ inside its dual space $\SOprime(G)$. We are therefore also interested in the following question:

\begin{enumerate}
\item[] \textit{Among the operators in $\Lin(\SO(G_{1}),\SOprime(G_{2}))$,
how do we characterize those having a kernel $K \in \SOprime(G_{1}\times G_{2})$ which is induced by a function in $\SO(G_{1}\times G_{2})$?}
\end{enumerate}
As it turns out, these will in fact be operators that map $\SOprime(G_{1})$ into $\SO(G_{2})$ in a certain way. We therefore have the immediate follow-up question:
\begin{enumerate}
\item[] \textit{Given an operator in $\Lin(\SO(G_{1}),\SOprime(G_{2}))$ that has its kernel $K\in \SOprime(G_{1}\times G_{2})$ induced by a function in $\SO(G_{1}\times G_{2})$, how do we extend its domain from $\SO(G_{1})$ to all of $\SOprime(G_{1})$?}
\end{enumerate}
The main results of this paper, Theorem \ref{th:new-inner-kernel-theorem} and  Theorem \ref{th:1608a}, answer these questions. For the formulation of our results, we need two auxiliary spaces:

\begin{definition} \label{def:the-spaces}
For  LCA groups $G_{1}$ and $G_{2}$  
we define the following two sets of operators:
\begin{align*}
 \mathcal{A}(G_{1},G_{2}) & = \{ A \in \Bil(\SOprime(G_{1})\times \SOprime(G_{2}),\C) : \, A  \ \text{is \ws \ continuous in each coordinate} \, \}, \\
\mathcal{B}(G_{1},G_{2}) & = \{ T \in \Lin(\SOprime(G_{1}),\SO(G_{2})) \, : \, T \ \text{maps norm bounded \ws convergent nets in $\SOprime(G_{1})$} \\ & \quad \quad \text{into norm convergent nets in $\SO(G_{2})$} \, \}. \end{align*}
\end{definition}

In Section \ref{sec:proof} we prove that the spaces $\mathcal{A}(G_{1},G_{2})$ and $\mathcal{B}(G_{1},G_{2})$ are complete with respect to their natural subspace topologies. Furthermore, we shall show that all elements in $\mathcal{B}(G_{1},G_{2})$ are \emph{nuclear/trace class} (and thus, in particular, also \emph{compact}) operators from $\SOprime(G_{1})$ into $\SO(G_{2})$ (see Section \ref{sec:examples}). 

\noindent 
We are  now ready to  formulate our first main result: 
\begin{theorem}[Inner kernel theorem for $\SO$] \label{th:new-inner-kernel-theorem} 
For  LCA groups $G_{1}$ and $G_{2}$ the four Banach spaces 
\[ \SO(G_{1}\times G_{2}), \ \mathcal{A}(G_{1},G_{2}), \ \mathcal{B}(G_{1},G_{2}) \ \text{and} \ \mathcal{B}(G_{2},G_{1})\] 
are naturally isomorphic. In particular, if any  of the four
\[ K\in\SO(G_{1}\times G_{2}), \, A\in \mathcal{A}(G_{1},G_{2}), \, T \in \mathcal{B}(G_{1},G_{2}) \ \text{or} \ S\in \mathcal{B}(G_{2},G_{1})\]  
is given, then the others are uniquely determined such that, 
for all $\sigma^{(i)}\in\SOprime(G_{i})$, $i=1,2$,
\begin{equation} \label{eq:1402a}(K,\sigma^{(1)}\otimes \sigma^{(2)})_{\SO,\SOprime(G_{1}\times G_{2})} = A(\sigma^{(1)},\sigma^{(2)}) = (T\sigma^{(1)},\sigma^{(2)})_{\SO,\SOprime(G_{2})} = (S\sigma^{(2)},\sigma^{(1)})_{\SO,\SOprime(G_{1})}.\end{equation}
\end{theorem}

\begin{remark} \label{rem:2601c} If the groups $G_{1}$ and $G_{2}$ are elementary, i.e., isomorphic to $\R^{n}\times \Z^{m} \times \mathbb{T}^{l} \times F$, where $F$ is some finite Abelian group and $l,n,m\in \N_{0}$, then a proof of  Theorem \ref{th:new-inner-kernel-theorem} is accessible using the isomorphism between $\SO$ and $\ell^{1}$ that is granted by the construction of Wilson bases. This approach to the inner kernel theorem can be found in \cite{feko98}. 
Our line of argumentation does not make use of this isomorphism and treats the general case of arbitrary LCA groups. We devote the entirety of Section \ref{sec:proof} to the proof of Theorem \ref{th:new-inner-kernel-theorem}.
\end{remark}

\begin{remark} \label{rem:2601a} Similar to the outer kernel theorem, given any $A\in \mathcal{A}(G_{1},G_{2}), \, T \in \mathcal{B}(G_{1},G_{2})$ or $S\in \mathcal{B}(G_{2},G_{1})$, the function $K\in \SO(G_{1}\times G_{2})$ satisfying \eqref{eq:1402a} is called the kernel of $A$, $T$ or $S$ and we denote this function by $\kappa(A)$, $\kappa(T)$ or $\kappa(S)$. \end{remark}

A combination of the inner and outer kernel theorem together with the continuous embedding of $\SO$ into $\SOprime$ (see Lemma \ref{le:SOprime-induced-by-SO}) allows us to make the following diagram for any two LCA groups $G_{1}$ and $G_{2}$. In the diagram the Hilbert-Schmidt operators from $\Lp{2}(G_{1})$ into $\Lp{2}(G_{2})$ are denoted by $\mathcal{HS}(G_{1},G_{2})$ .
\[ \setlength\arraycolsep{4pt} \def\arraystretch{1.3}  \begin{matrix} 
\multicolumn{3}{l}{ \text{\footnotesize{Inner Kernel Theorem}} } & & \\
\cline{1-3}
\multicolumn{1}{|c}{\mathcal{A}(G_{1},G_{2})} & \cong & \multicolumn{1}{c|}{\SO(G_{1}\times G_{2})} & & \\ 
\cline{2-3}
\multicolumn{1}{|c|}{\cong} & & & & \\
\multicolumn{1}{|c|}{\mathcal{B}(G_{1},G_{2})} & \bigsubseteq & \Lin(\SOprime(G_{1}),\SO(G_{2})) & \bigsubseteq & \Lin(\Lp{2}(G_{1}),\SO(G_{2})) & \bigsubseteq & \Lin(\SO(G_{1}),\SO(G_{2})) \\
\cline{1-1}
& & \rotsubseteq & & \rotsubseteq & & \rotsubseteq \\
& & \Lin(\SOprime(G_{1}),\Lp{2}(G_{2}))
& \bigsubseteq & \Lin(\Lp{2}(G_{1}),\Lp{2}(G_{2})) & \bigsubseteq & \Lin(\SO(G_{1}),\Lp{2}(G_{2})) \\
& & & & \rotasubseteq \phantom{MMM} & & & \\
& & \multicolumn{3}{c}{
\begin{matrix} \cline{3-3} \rotsubseteq & \ & \multicolumn{1}{|c|}{\Lp{2}(G_{1}\times G_{2})\cong\mathcal{HS}(G_{1},G_{2})} & & \ & \rotsubseteq \\ \cline{3-3} \end{matrix}
} & & \rotsubseteq \\
& & \multicolumn{3}{c}{
\begin{matrix} & \ & \text{\footnotesize{Hilbert-Schmidt Operators}} & & \end{matrix}
} & & \\[4pt]
\cline{7-7}
& & \Lin(\SOprime(G_{1}),\SOprime(G_{2})) & \bigsubseteq & \Lin(\Lp{2}(G_{1}),\SOprime(G_{2})) & \bigsubseteq & 
\multicolumn{1}{|c|}{\Lin(\SO(G_{1}),\SOprime(G_{2}))} \\  
& & & & & & \multicolumn{1}{|c|}{\cong} \\
\cline{5-6}
& & & & \multicolumn{1}{|c}{\Bil(\SO(G_{1})\times\SO(G_{2}),\C)} & \cong & \multicolumn{1}{c|}{\SOprime(G_{1}\times G_{2})} \\
\cline{5-7}
& & & & \multicolumn{3}{r}{ \text{\footnotesize{Outer Kernel Theorem}} }
\end{matrix} \]
Furthermore, we have the following inclusions for Banach spaces of operators: 
\begin{equation} \label{eq:op-BGT} \mathcal{B}(G_{1},G_{2}) \subseteq \mathcal{HS}(G_{1},G_{2}) \subseteq \Lin(\SO(G_{1}),\SOprime(G_{2})).\end{equation}
In fact, these three spaces form a Banach Gelfand triple and have been investigated in \cite{ba10-2},\cite{cofelu08} and \cite{feko98}. For appliications of th \cite{fe09} for applications.

\begin{remark} \label{rem:2601b} Both the inner and outer kernel theorem for $\SO$ are analogous to the situation for nuclear spaces, cf.\ Chapter 50 and 51 in Tr\`eves book  \cite{tr67}. Further references to the theory of nuclear spaces and their kernel theorems are Delcroix \cite{de10-1} and H\"ormander \cite{ho03}. \end{remark}
Note that $\SO$ contains the Schwartz(-Bruhat) space as a dense subspace (\cite[Theorem 9]{fe81-2}) and that $\SOprime$ is a subspace of the tempered distributions. For more on the Schwartz-Bruhat functions we refer to the original literature \cite{br61,os75}.
 
The paper is structured as follows. 
Section \ref{sec:preliminaries} recollects necessary facts about the function space $\SO(G)$ and its continuous dual space $\SOprime(G)$. Section \ref{sec:three} is comprised of several smaller pieces. The first of which, Section \ref{sec:nets}, states when the continuity of the operators in the spaces $\mathcal{A}$ and $\mathcal{B}$ can be described with the notion of sequences rather than that of nets. Section \ref{sec:kernel-in-s0} contains the second main result of this paper, Theorem \ref{th:1608a}. This result gives a more quantitative description of the operators in $\Lin(\SO(G_{1}),\SOprime(G_{2}))$ that have a kernel in $\SO$ and establishes a more natural norm on those operators (rather than the subspace topologies as mentioned following Definition \ref{def:the-spaces}). Section \ref{sec:linear-albera} shows similarities between the matrix representation of operators between finite dimensional spaces and the space $\mathcal{B}(G_{1},G_{2})$. Examples of operators with kernel in $\SO$ and results concerning series representations, nuclearity and trace-class properties of the operators in $\mathcal{B}$ are shown in Section \ref{sec:examples}. In Section \ref{sec:reg-app-id} we define and show examples of what we call regularizing approximations of the identity. Finally, Section \ref{sec:kernel-for-modspaces} contains some comments on extensions of the theory and references to related work. As mentioned earlier, Section \ref{sec:proof} is solely concerned with the proof of the Theorem \ref{th:new-inner-kernel-theorem}.

\section{Preliminaries} \label{sec:preliminaries}
\subsection{Harmonic analysis on LCA groups}
Throughout the paper we will be working with locally compact Abelian Hausdorff groups, which we denote by $G$, $G_{i}$, $i = 1,2,\ldots$. As any locally compact group, an LCA groups carries an (up to scaling) unique translation invariant measure, the \emph{Haar measure}. The \emph{dual group} of an LCA group $G$ is the multiplicative group of all continuous group homomorphisms from $G$ into the torus $\{ z\in\C \, : \, \vert z \vert =1\}$, which we denote by  $\ghat$. Under the topology of uniform convergence on compact sets the dual group becomes an LCA group itself. As such, it  carries a Haar measure. Without loss of generality we always assume that these measures are normalized such that 
\[ f(x) = \int_{\ghat} \hat{f}(\omega) \, \omega(x) \, d\mu_{\ghat}(\omega) \ \ \text{for almost every} \ x\in G\]
for all $f\in \Lp{1}(G)$ with $\hat{f}\in \Lp{1}(\ghat)$,
where $\hat{f}$ is the Fourier transform of $f$, $\hat{f}(\omega) = \int_{G} f(x) \, \overline{\omega(x)} \, d\mu_{G}(x)$, $\omega\in\ghat$. 
Typically we will perform integration in the time-frequency plane (phase space) $G\times\ghat$ so that we encounter integrals of the form $\int_{G\times\ghat} f(\TFelement) \, d\mu_{G\times\ghat}(\TFelement)$
for suitable complex valued functions $f$ on $G\times\ghat$. From now on we shall  simplify the  notation and write $\int_{G} \ldots \, dx$, $\int_{\ghat} \ldots \, d\omega$, and $\int_{G\times\ghat} \ldots \, d\TFelement$, rather than, e.g., $\int_{G\times\ghat} \ldots \, d\mu_{G\times\ghat}(\TFelement)$. For more on integration on locally compact groups and abstract harmonic analysis we refer to, e.g., \cite{fo16,na65-2} and \cite{rest00}.

\subsection{The space $\SO$} \label{sec:S0}
In this section we summarize results on the space $\SO$ and its dual space $\SOprime$. As we often will deal with functions in the spaces $\SO(G_{1})$ and $\SO(G_{2})$ and as well as with generalized functions in $\SOprime(G_{1})$ and $\SOprime(G_{2})$ for typically different locally compact Abelian groups $G_{i}$, $i=1,2$, we define once and for all that $f^{(i)}$ and $\sigma^{(i)}$ denote a function and a generalized function in $\SO(G_{i})$ and $\SOprime(G_{i})$, respectively. Different functions in $\SO(G_{i})$ will be denoted either by different letters, e.g., $f^{(i)},g^{(i)}$ and $h^{(i)}$, or with an index, $f^{(i)}_{j}$. 

For functions in $\SO(G_{1})$ and $\SO(G_{2})$ the tensor product
\[ \big( f^{(1)} \otimes f^{(2)} \big)(x^{(1)},x^{(2)}) = f^{(1)}(x^{(1)}) \cdot f^{(2)}(x^{(2)}), \ \ (x^{(1)},x^{(2)})\in G_{1}\times G_{2},\]
is a bilinear and bounded operator into $\SO(G_{1}\times G_{2})$. In fact, 
\[ \Vert f^{(1)}\otimes f^{(2)} \Vert_{\SO(G_{1}\times G_{2}),g^{(1)}\otimes g^{(2)}} = \Vert f^{(1)} \Vert_{\SO(G_{1}),g^{(1)} } \cdot \Vert f^{(2)} \Vert_{\SO(G_{2}),g^{(2)} }.\]
Any $f\in \SO(G_{1}\times G_{2})$ can be written (in a non-unique way) as an infinite sum for appropriately chosen sequences $(f_{j}^{(i)})_{j\in \N}$ in $\SO(G_{i})$, $i=1,2$,
\begin{equation}
\label{eq:1508a}
 f = \sum_{j\in \N} f^{(1)}_{j} \otimes f^{(2)}_{j}
\ \ \text{such that} \ \ \sum_{j\in\N} \Vert f^{(1)}_{j} \Vert_{\SO} \, \Vert f^{(2)}_{j}\Vert_{\SO} < \infty, \end{equation}
where the sum is absolutely norm convergent in $\SO(G_{1}\times G_{2})$.
Moreover, the $\SO(G_{1}\times G_{2})$-norm is equivalent to the projective tensor product norm
\begin{equation}
\label{eq:1508b}
 \Vert f \Vert_{\hat{\otimes}} = \inf \big\{ \sum_{j\in \N} \Vert f^{(1)}_{j} \Vert_{\SO} \, \Vert f^{(2)}_{j} \Vert_{\SO} \big\}, 
 \end{equation}
where the infimum is taken over all admissible representations of $f$ as in \eqref{eq:1508a}. We thus have the following.
\begin{lemma}\label{le:SO-tensor-factorization} 
Given LCA groups  $G_{1}$ and $G_{2}$  one has 
$\SO(G_{1}\times G_{2}) = \SO(G_{1})\hat{\otimes} \SO(G_{2})$.
\end{lemma}
\noindent These statements were originally proven in \cite[Theorem 7]{fe81-2} and can also be found in \cite[Theorem 7.4]{ja19}. \\

The translation operator $T_{x}$ and the modulation operator $E_{\omega}$ are given by 
\[ T_{x} f(t) = f(t-x) \ \ \text{and} \ \ E_{\omega} f(t) = \omega(t) f(t), \ \ t,x\in G, \, \omega\in \ghat.\]
They act as linear and isometric operators on $\SO(G)$ and so do time-frequency shift operators: 
\[ \pi(\TFelement) = \pi(x,\omega) = E_{\omega} T_{x}  \quad \mbox{for} \,\,  \nu = 
(x , \omega) \in G \times \ghat.
\]

Besides the definition of $\SO$ in the introduction, there is also an atomic characterization:
\begin{lemma} \label{le:S0-Gabor-expansion}
Fix a non-zero function $g\in \SO(G)$. For any $f\in\SO(G)$ there exists a sequence $c\in \ell^{1}(\N)$ and elements $\TFelement_{j}\in G\times\ghat$, $j\in\N$ such that $f = \sum_{j\in\N} c_{j} \, \pi(\TFelement_{j}) g$. Furthermore, $\Vert f \Vert = \inf\,\Vert c \Vert_1$, where the infimum is taken over all admissible representations of $f$ as above, defines an equivalent norm on $\SO(G)$.
\end{lemma}
\noindent This result goes back to  \cite{fe87-1} and can also be found in \cite[Theorem 7.2]{ja19}. 

The dual space $\SOprime(G)$ is a Banach space with respect to the usual operator topology 
\begin{equation} \label{eq:operator-norm-on-SOprime} \Vert \sigma \Vert_{\SOprime(G),g} = \sup_{f\in \SO(G)\backslash \{0\}} \frac{\vert (f,\sigma)_{\SO,\SOprime(G)}\vert}{\Vert f \Vert_{\SO(G),g}}, \ \sigma\in \SOprime(G),\end{equation}
where $g\in \SO(G)$ is any non-zero function.
There is another indispensable norm on $\SOprime(G)$.
\begin{lemma}[{see \cite[Proposition 6.11]{ja19}}] \label{le:STFT-norm-on-SOprime} For any $g\in \SO(G)\backslash\{0\}$
\[ \Vert \cdot \Vert_{\mathbf{M}^{\infty}_g} : \SOprime(G) \to \R_{0}^{+}, \ \Vert \sigma \Vert_{\mathbf{M}^{\infty}_g} = \sup_{\TFelement\in G\times\ghat} \vert ( \pi(\TFelement) g, \sigma)_{\SO,\SOprime(G)} \, \vert \]
is a norm on $\SOprime(G)$ which is equivalent to the norm in \eqref{eq:operator-norm-on-SOprime}. 
\end{lemma}

For alternative recent approaches to this space, whose elements are now called {\it ``mild distribution''} see 
\cite{fe19} and \cite{fe20-1}.

In many situations the norm convergence in $\SOprime$ is too strong and therefore we also have to make use of the {\it  \ws -topology}. Recall 
that $\sigma_0\in\SOprime(G)$ is the \ws \ limit of a net $(\sigma_\alpha)$ in $\SOprime(G)$ 
 if 
\[ \lim_{\netparam} \vert (f, \sigma_{\netparam}-\sigma_{0})_{\SO,\SOprime(G)} \vert = 0 \ \  \mbox{for any} \ f \in \SO(G).\]

As for every Banach space (see \cite[p.\ 98]{me98-1}), also for $\SO(G)$ the Hahn-Banach Theorem provides 
an isometric embedding  into its double dual $\SOdoubleprime(G)$ via the canonical embedding
\[ \iota : \SO(G)\to \SOdoubleprime(G), \ \iota(f) = \sigma \mapsto (f,\sigma)_{\SO,\SOprime(G)}, \ f\in \SO(G), \, \sigma\in \SOprime(G). \]
Moreover, $\iota(\SO(G))$ is exactly the set of all bounded \ws \ continuous functionals on $\SOprime(G)$. That is, a linear and bounded functional $\varphi: \SOprime(G) \to \C$ sends bounded \ws \ convergent nets in $\SOprime(G)$ into norm convergent nets in $\C$ if and only if $\varphi$ is of the form $\varphi(\sigma) = (f,\sigma)_{\SO,\SOprime(G)}$ for some $f\in \SO(G)$ (see \cite[Proposition 2.6.4]{me98-1}). Henceforth we view, if necessary, $\SO(G)$ as a closed subspace of $ \SOdoubleprime(G)$.  This fact is essential for our proof of Theorem \ref{th:new-inner-kernel-theorem} in Section \ref{sec:proof}.

Similar as for functions, we can define the tensor product $\sigma^{(1)}\otimes \sigma^{(2)}$ of two generalized functions $\sigma^{(1)}\in \SOprime(G_{1})$ and $\sigma^{(2)}\in \SOprime(G_{2})$. It is the unique element in $\SOprime(G_{1}\times G_{2})$ with the property that
\begin{equation}
( f^{(1)} \otimes f^{(2)} , \sigma^{(1)} \otimes \sigma^{(2)} )_{\SO,\SOprime(G_{1}\times G_{2}) } = ( f^{(1)} , \sigma^{(1)} )_{\SO,\SOprime(G_{1})} \, ( f^{(2)}, \sigma^{(2)} )_{\SO,\SOprime(G_{2})},
\end{equation} 
for all $f^{(i)}\in \SO(G_{i})$, $i=1,2$.
One can show that
\begin{equation} \label{eq:soprime-tensor-norm} \Vert \sigma^{(1)} \otimes \sigma^{(2)} \Vert_{\mathbf{M}^{\infty}_{g^{(1)}\otimes g^{(2)}}} = \Vert \sigma^{(1)} \Vert_{\mathbf{M}^{\infty}_{g^{(1)}}} \, \Vert \sigma^{(2)} \Vert_{\mathbf{M}^{\infty}_{g^{(2)}}}.\end{equation}
For a proof of this we refer to \cite[Corollary 9.2]{ja19}.

As mentioned in the introduction, the space $\SO(G)$ is embedded into its dual space $\SOprime(G)$ in a very natural way. In order to properly formulate this result we define the \emph{modulation space} (for the parameter $1$) as the subspace of $\SOprime(G)$ given by
\begin{equation}
\label{Midef1}
\Misp(G) = \Big\{ \sigma\in \SOprime(G) \, : \, \int_{G\times\ghat} \vert ( \pi(\TFelement) g, \sigma)_{\SO,\SOprime(G)} \vert \, d\TFelement < \infty \Big\}, 
\end{equation}
where $g$ is some non-zero function in $\SO(G)$. In Section \ref{sec:kernel-for-modspaces} we give references to literature on the modulation spaces.
 The norm
\begin{equation}
\label{Minorm1}
\Vert \cdot \Vert_{\Misp_{g}} :  \Misp(G) \to \R_{0}^{+}, \quad \Vert \sigma \Vert_{\Misp_{g}} = \int_{G\times\ghat} \vert ( \pi(\TFelement) g, \sigma)_{\SO,\SOprime(G)} \vert \, d\TFelement 
\end{equation}
turns $\Misp(G)$ into a Banach space. 
Each function $g\in \SO(G)\backslash\{0\}$ induces an equivalent norm on $\Misp(G)$. 
One can show that there exists a constant $c>0$ such that $\Vert \sigma \Vert_{\SOprime} \le c \, \Vert \sigma \Vert_{\Misp}$ for all $\sigma\in \Misp(G)$. That is, $\Misp(G)$ is continuously embedded into $\SOprime(G)$.

\begin{lemma}\label{le:SOprime-induced-by-SO} The Banach spaces $\SO(G)$ and $\Misp(G)$ are naturally isomorphic. In particular: 
\begin{enumerate} 
\item[(i)] Via the Haar measure on $G$  every  $h\in \SO(G)$ 
induces a (unique) functional $\sigma_h \in \SOprime(G)$:  
\begin{equation} \label{eq:regular-distribution} (f,\sigma_h)_{\SO,\SOprime(G)} = \int_{G} f(t) \, h(t) \, dt \ \ \text{for all} \ \ f\in\SO(G).\end{equation}
This embedding of $\SO(G)$ into $\SOprime(G)$ is linear, continuous and injective.
\item[(ii)] If $\sigma$ is a generalized function in $\SOprime(G)$, then there exists a function $h\in \SO(G)$ such that \eqref{eq:regular-distribution} holds if and only if $\sigma\in \Misp(G)$. The function $h \in \SO(G)$ is characterized by the fact that for some $g\in\SO(G)\backslash\{0\}$ (and then for every such $g$) 
one has:  
\begin{equation}
 (h,\tilde{\sigma})_{\SO,\SOprime(G)} = \Vert g \Vert_{2}^{-2} \int_{G\times\ghat}
 \big( \overline{\pi(\TFelement) g}, \sigma \big) \, \big( \pi(\TFelement)g,\tilde{\sigma}\big) \, d\TFelement \ \ \text{for all} \ \ \tilde{\sigma}\in\SOprime(G).
 \end{equation}
\end{enumerate} 
\end{lemma}
One can verify that the embeddings in Lemma \ref{le:SOprime-induced-by-SO}(i) and (ii) are inverses of one another (independently of the choice of the function $g$ in (ii)). The details can be found in \cite[Theorem 6.12]{ja19}. A side note: if $h$ is any function in $\Lp{p}(G)$, $p\in[1,\infty]$, then $h$ also induces a functional in $\SOprime(G)$ as in \eqref{eq:regular-distribution}. 

By the natural isomorphism between $\SO(G)$ and $\Misp(G)$ the function space $\SO(G)$ is continuously embedded into its dual space $\SOprime(G)$. Due to this relation between $\SO(G)$ and $\SOprime(G)$ we allow ourselves, for all $f,h\in \SO(G)$, to write $(f,h)_{\SO,\SOprime(G)}$,
by which we mean the action that the function $h$ has on $f$ as in Lemma \ref{le:SOprime-induced-by-SO}(i). 
Note that $(f,h)_{\SO,\SOprime(G)} = (h,f)_{\SO,\SOprime(G)}$.

It turns out that the just mentioned embedding places $\SO$ inside $\SOprime$ as a \ws \ dense space. Actually, we have the following.
\begin{lemma}[{see \cite[Proposition 6.15]{ja19}}] \label{le:bounded-S0prime-approx} For any $\sigma\in\SOprime(G)$ there exists a net 
 $(\sigma_{\netparam}) $ in $\Misp(G)\cong \SO(G)$
such that
\[ \lim_{\netparam} \big\vert ( f, \sigma-\sigma_{\netparam})_{\SO,\SOprime} \vert = 0 \ \ \text{for all} \ \ f\in\SO(G) \ \ \text{and such that} \ \ \Vert \sigma_{\netparam} \Vert_{\SOprime} \le \Vert \sigma \Vert_{\SOprime}.\]
\end{lemma}

The translation and modulation operators can be uniquely extended from operators on $\SO(G)$ to \ws-\ws \ continuous operators on $\SOprime(G)$. We will denote these extensions by the same symbol. Specifically, for $f\in \SO(G)$, $\sigma\in \SOprime(G)$ and $\TFelement = (x,\omega) \in G\times\ghat$, 
they are characterized by the following identities: 
\begin{align*} ( f , T_{x} \sigma)_{\SO,\SOprime(G)} & = (T_{-x} f, \sigma)_{\SO,\SOprime(G)}, \\
(f, E_{\omega} \sigma)_{\SO,\SOprime(G)} & = (E_{\omega} f, \sigma)_{\SO,\SOprime(G)}, \\
( f, \pi(x,\omega) \,\sigma)_{\SO,\SOprime(G)} & = \omega(x) \, ( \pi(-x,\omega) f , \sigma)_{\SO,\SOprime(G)}.\end{align*}
In addition, for $g, h \in\SO(G)$, we define
\begin{align*} ( f, h \cdot \sigma)_{\SO,\SOprime(G)} & = (f\cdot h, \sigma)_{\SO,\SOprime(G)}, \\
(f, g  \ast \sigma)_{\SO,\SOprime(G)} & = (f \ast g^{\checkmark}\!, \sigma)_{\SO,\SOprime(G)} , \ \ g^{\checkmark}\!(t) = g(-t), \ t\in G. 
\end{align*}
These formulas remain valid for $h$ being a pointwise multiplier of 
$\SO(G)$ or $g$ having a Fourier transform with this property (defining a bounded convolution operator on $\SO(G)$). 

The complex conjugation of a generalized function is defined by the relation
\[ (f, \overline{{\sigma}} )_{\SO,\SOprime(G)} = \overline{(\overline{f}, \sigma)}_{\SO,\SOprime(G)},\]
The reader may verify that these definitions are compatible with the embedding of $\SO(G)$ into $\SOprime(G)$ as described in Lemma \ref{le:SOprime-induced-by-SO}
and are in fact uniquely determined based on this consistency consideration. 

Observe that the extension of the translation operator to $\SOprime(G)$ is \emph{not} the same as its Banach space adjoint, which, by definition, is the operator given by
\begin{align*}
(T_{x})^{\times} : \SOprime(G) \to \SOprime(G), \ ( f, (T_{x})^{\times} \sigma)_{\SO,\SOprime(G)} = ( T_{x} f, \sigma)_{\SO,\SOprime(G)}.
\end{align*}
However, it so happens that the Banach space adjoint of the modulation operator $E_{\omega} : \SO(G) \to \SO(G)$ is the same as its unique extension to an operator on $\SOprime(G)$.  

Throughout the paper $\langle \cdot , \cdot \rangle$ is the $\Lp{2}$-inner product (with the anti-linearity in the second entry), which is well-defined for functions in $\SO(G)$ as $\SO(G)\subseteq \Lp{2}(G)$. In fact, $\SO(G)$ is continuously embedded into all the $\Lp{p}(G)$ spaces: for all $p\in [1,\infty]$ and $f\in \SO(G)$,
\[ \Vert f \Vert_{p} \le \Vert g \Vert_{q}^{-1} \, \Vert f \Vert_{\SO(G),g}, \]
where $p^{-1}+q^{-1} = 1$ for $p\in (1,\infty)$ and the usual convention if $p=1$ or $p=\infty$ (this follows from \cite[Lemma 4.19]{ja19}). Furthermore $\SO(G)$ is continuously embedded into $\mathbf{C}_{0}(G)$ and hence $\SOprime(G)$ contains the Dirac delta distribution $\delta_{x}: f \mapsto f(x)$, $x\in G$, $f\in\SO(G)$. 

We will make frequent use of the following equality.
\begin{lemma}[{see \cite[Lemma 6.10(iv)]{ja19}}] \label{le:STFT-SO-SOprime} If $g\in\SO(G)\backslash\{0\}$, then for any $f\in \SO(G)$ and $\sigma\in \SOprime(G)$
\begin{equation} \label{eq:1412a} (f,\sigma)_{\SO,\SOprime(G)} = \Vert g \Vert_{2}^{-2} \int_{G\times\ghat} \langle f, \pi(\TFelement) g\rangle \, (\pi(\TFelement) g, \sigma)_{\SO,\SOprime(G)} \, d\TFelement.\end{equation}
\end{lemma}
Lastly, we define the short-time Fourier transform with respect to a function $g\in\SO(G)$ to be the operator
\[ \mathcal{V}_{g} : \SOprime(G)\to \mathbf{C}_{b}(G\times\ghat) , \ \mathcal{V}_{g}\sigma(\TFelement) = \big( \overline{\pi(\TFelement) g} , \sigma\big)_{\SO,\SOprime(G)} \ \ \text{for all} \ \ \sigma\in\SOprime(G), \ \TFelement\in G\times\ghat. \] 
The operator maps $\Lp{2}(G)$ into $\Lp{2}(G\times\ghat)$ and it maps $\SO(G)$ into $\SO(G\times\ghat)$ (see \cite[Section 6]{feko98} or \cite[Theorem 5.3(ii)]{ja19}. Note that if $f\in \Lp{2}(G)$, then $\mathcal{V}_{g}f (\TFelement) =\langle f, \pi(\TFelement) g\rangle$, $\TFelement\in G\times\ghat$. Using the short-time Fourier transform we can reformulate \eqref{eq:1412a} as
$$\Vert g\Vert_{2}^{2} \ (f,\overline{\sigma})_{\SO,\SOprime(G)} = \int_{G\times\ghat} \mathcal{V}_{g}f(\TFelement) \overline{\mathcal{V}_{g}\sigma(\TFelement)} \, d\TFelement. $$ 

\section{Operators that have a kernel in $\SO$}
\label{sec:three}
\subsection{Nets versus sequences}
\label{sec:nets}

The spaces of operators that are identified with $\SO(G_{1}\times G_{2})$ by Theorem \ref{th:new-inner-kernel-theorem} are 
defined using \ws \ 
continuity in $\SOprime$. The \ws \ topology on $\SOprime$ is non-metrizable (unless $\SO$ is finite dimensional, \cite[Proposition 2.6.12]{me98-1}) and it is therefore properly described using nets. However, in some cases, e.g., if $G=\R^{d}$, we may use the notion of sequences to describe the spaces $\mathcal{A}$ and $\mathcal{B}$.

\begin{lemma} If $G_{1}$ and $G_{2}$ are $\sigma$-compact and metrizable, then the Banach spaces $\mathcal{A}(G_{1},G_{2})$ and $\mathcal{B}(G_{1},G_{2})$ can be described by the behavior of convergent sequences. Specifically, 
\begin{align*}
    \mathcal{A}(G_{1},G_{2}) & = \{ A \in \Bil(\SOprime(G_{1})\times \SOprime(G_{2}),\C) \, : \\ & \qquad \qquad A \textnormal{ is sequentially \ws \ continuous in each coordinate } \} \\
    \mathcal{B}(G_{1},G_{2}) & = \{ T \in \Lin(\SOprime(G_{1}),\SO(G_{2})) \, : \\ & \qquad \qquad  T \textnormal{ maps \ws-convergent sequences in $\SOprime(G_{1})$} \\
    & \qquad \qquad \textnormal{into norm convergent sequences in $\SO(G_{2})$ } \}
\end{align*}


\end{lemma}
\begin{proof} If a locally compact Abelian group $G$ is $\sigma$-compact and metrizable then also its dual group $\ghat$ is $\sigma$-compact and metrizable \cite[Section 3]{boro15}. It is a fact that $\SO$ can be described as a coorbit space associated to the Heisenberg representation of $G\times \ghat$ \cite{fegr92-1}. Coorbit theory \cite[Theorem 6.1]{fegr89}, together with the fact that the time-frequency plane $G\times \ghat$ is $\sigma$-compact, implies the separability of $\SO(G)$. Thus, by the assumption in the lemma, the spaces $\SO(G_{i})$, $i=1,2$ are separable. The Banach-Alaoglu theorem thus implies that the \ws \ topology on $\SOprime$ on any bounded set is metrizable. Hence the notions of continuity by \emph{bounded} convergent nets and convergent sequences coincide. 
\end{proof}

The commonly used  locally compact Abelian groups $\R$, $\Z$, $\mathbb{T}$, $\Z/N\Z$ $N=1,2,\ldots$ and the $p$-adic numbers are $\sigma$-compact and metrizable. The additive group $\R$ under the discrete topology is an example of a non-$\sigma$-compact (albeit metrizable) locally compact Abelian group.

\subsection{Identifying operators that have a kernel in $\SO$} 
\label{sec:kernel-in-s0}

In this section we answer the second and third question posed in the Introduction, which we expand on here. 

Let $T$ be an operator in $\Lin(\SO(G_{1}),\SOprime(G_{2}))$. By the outer kernel theorem $T$ has a kernel $K$ in $\SOprime(G_{1}\times G_{2})$. 
Assume now that this kernel is induced by a function in $\SO(G_{1}\times G_{2})$. By the inner kernel theorem we know that these operators are exactly the ones that belong to $\mathcal{B}(G_{1},G_{2})\subseteq \Lin(\SOprime(G_{1}),\SO(G_{2}))$. For such an operator $T: \SO(G_{1}) \to \SOprime(G_{2})$ we are faced with the following questions.
\begin{itemize}
    \item[(a)] How do we verify that the domain of the operator $T$ can be extended from $\SO(G_{1})$ to $\SOprime(G_{1})$?
    \item[(b)] How do we know that its co-domain actually is $\SO(G_{2})$ rather than $\SOprime(G_{2})$?
    \item[(c)] How can we verify its continuity properties as described in Definition \ref{def:the-spaces}?
\end{itemize} 
Naturally, the same questions can be formulated for operators $A\in \Bil(\SO(G_{1})\times\SO(G_{2}),\C)$ whose kernel might be induced by a function in $\SO(G_{1}\times G_{2})$. 

The following theorem characterizes the operators in $\Lin(\SO(G_{1}),\SOprime(G_{2}))$ and $\Bil(\SO(G_{1})\times\SO(G_{2}),\C)$ that have a kernel in $\SO(G_{1}\times G_{2})$ and it describes how their domain extends from $\SO$ to $\SOprime$. 
\begin{theorem} \label{th:1608a} For $i=1,2$ fix a function $g^{(i)}\in\SO(G_{i})\backslash\{0\}$ such that $\Vert g^{(i)} \Vert_{2} = 1$. 
\begin{enumerate}
\item[(i)] If $A$ is an operator in $\Bil(\SO(G_{1})\times \SO(G_{2}),\C)$, then its kernel $\kappa(A)\in \SOprime(G_{1}\times G_{2})$ is induced by a function in $\SO(G_{1}\times G_{2})$, i.e.\ $A\in\mathcal{A}(G_{1},G_{2})$, if and only if 
\begin{equation} \label{eq:norm-A}  \int\limits_{\substack{G_{1}\times\ghat_{1}  \times G_{2}\times\ghat_{2}}} \hspace{-0.5cm} \big\vert A\big(\pi(\TFelement^{(1)}) g^{(1)}, \pi(\TFelement^{(2)})g^{(2)}\big) \big\vert \, d(\TFelement^{(1)},\TFelement^{(2)}) < \infty.\end{equation}
In that case the operator $A:\SOprime(G_{1})\times\SOprime(G_{2})\to\C$ satisfies
\begin{align} & A(\sigma^{(1)},\sigma^{(2)}) \nonumber \\
& = \int\limits_{\substack{G_{1}\times \ghat_{1}  \times G_{2}\times \ghat_{2}}} \hspace{-0.5cm} \mathcal{V}_{g^{(1)}}\sigma^{(1)}(\TFelement^{(1)})  \cdot \mathcal{V}_{g^{(2)}}\sigma^{(2)}(\TFelement^{(2)}) \cdot  A\big( \pi(\TFelement^{(1)})g^{(1)} , \pi(\TFelement^{(2)})g^{(2)}\big) \ d(\TFelement^{(1)},\TFelement^{(2)}). \label{eq:integral-A} \end{align}
\item[(ii)] If $T$ is an operator in $\Lin\big(\SO(G_{1}),\SOprime(G_{2})\big)$, then its kernel  $\kappa(T)\in \SOprime(G_{1}\times G_{2})$ is induced by a function in $\SO(G_{1}\times G_{2})$, i.e.\ $T\in\mathcal{B}(G_{1},G_{2})$, if and only if
\begin{equation} \label{eq:norm-B} \int\limits_{\substack{G_{1}\times\ghat_{1} \times G_{2}\times\ghat_{2}}} \hspace{-0.5cm} \big\vert \big( \pi(\TFelement^{(2)}) g^{(2)}, T \circ \pi(\TFelement^{(1)})g^{(1)}\big)_{\SO,\SOprime(G_{2})} \big\vert \, d(\TFelement^{(1)},\TFelement^{(2)}) < \infty.\end{equation}
In that case the operators $T:\SOprime(G_{1})\to\SO(G_{2})$ satisfies
\end{enumerate}
\begin{align} & (T\sigma^{(1)},\sigma^{(2)})_{\SO,\SOprime(G_{2})}  \nonumber \\
& = \int\limits_{\substack{G_{1}\times \ghat_{1}  \times G_{2}\times \ghat_{2}}} \hspace{-0.5cm} \mathcal{V}_{g^{(1)}}\sigma^{(1)}(\TFelement^{(1)}) \cdot \mathcal{V}_{g^{(2)}}\sigma^{(2)}(\TFelement^{(2)}) \cdot  \big( \pi(\TFelement^{(2)})g^{(2)}, T\pi(\TFelement^{(1)})g^{(1)} \big)_{\SO,\SOprime} \ d(\TFelement^{(1)},\TFelement^{(2)}).\label{eq:integral-B} \end{align}
\end{theorem}
\begin{remark} The formula in \eqref{eq:integral-A} extends the domain of $A$ from $\SO(G_{1})\times\SO(G_{2})$ to $\SOprime(G_{1})\times\SOprime(G_{2})$, and \eqref{eq:integral-B} extends the domain of $T$ from $\SO(G_{1})$ to $\SOprime(G_{1})$.
\end{remark}
\begin{remark}The condition in Theorem \ref{th:1608a} that $\Vert g^{(i)}\Vert_{2} = 1$ is only necessary to make the equalities in \eqref{eq:integral-A} and \eqref{eq:integral-B} more pleasant. Otherwise the integrals need to be normalized by $\Vert g^{(1)}\otimes g^{(2)} \Vert_{2}^{-2}$, see the details in the proof.
\end{remark}

\begin{proof}[Proof of Theorem \ref{th:1608a}]
We will only prove (i) as the proof of (ii) is similar. By Theorem \ref{th:outer-kernel-theorem} and by assumption we know that $A$ has a kernel $\kappa(A)\in \SOprime(G_{1}\times G_{2})$ so that
\begin{align*}
& \quad \  \int\limits_{\substack{G_{1}\times\ghat_{1}\times G_{2}\times\ghat_{2}}} \hspace{-0.5cm} \big\vert A\big(\pi(\TFelement^{(1)}) g^{(1)}, \pi(\TFelement^{(2)})g^{(2)}\big) \big\vert \, d(\TFelement^{(1)},\TFelement^{(2)}) \\
& = \int\limits_{\substack{G_{1}\times\ghat_{1} \times G_{2}\times\ghat_{2}}} \hspace{-0.5cm} \big\vert \big(\pi(\TFelement^{(1)}) g^{(1)}\otimes\pi(\TFelement^{(2)})g^{(2)} , \kappa(A) \big)_{\SO,\SOprime(G_{1}\times G_{2})} \big\vert \, d(\TFelement^{(1)},\TFelement^{(2)}). \\
& = \int\limits_{\substack{G_{1}\times\ghat_{1} \times G_{2}\times\ghat_{2}}} \hspace{-0.5cm} \big\vert \big(E_{\omega^{(1)},\omega^{(2)}} T_{x^{(1)},x^{(2)}} ( g^{(1)}\otimes g^{(2)}) , \kappa(A) \big)_{\SO,\SOprime(G_{1}\times G_{2})} \big\vert \, d(x^{(1)},\omega^{(1)},x^{(2)},\omega^{(2)}).
\end{align*}
By Lemma \ref{le:SOprime-induced-by-SO} the last integral is finite if and only if the generalized function $\kappa(A)\in \SOprime(G_{1}\times G_{2})$ is induced by a (unique) function in $\SO(G_{1}\times G_{2})$, which we shall also call $\kappa(A)$.  By Theorem \ref{th:new-inner-kernel-theorem}  this kernel is identifiable with an operator $A\in \mathcal{A}\subseteq \Bil(\SOprime(G_{1})\times\SOprime(G_{2}),\C)$ which satisfies
\[ A(\sigma^{(1)},\sigma^{(2)}) = ( \kappa(A) , \sigma^{(1)}\otimes\sigma^{(2)})_{\SO,\SOprime(G_{1}\times G_{2})}.\]
By use of Lemma \ref{le:STFT-SO-SOprime} (with $g= \overline{g^{(1)}\otimes g^{(2)}}$, $f=\kappa(A)$, $\sigma=\sigma^{(1)}\otimes\sigma^{(2)}$) we can establish the desired equality.
\begin{align*}
& \quad \ ( \kappa(A) , \sigma^{(1)}\otimes\sigma^{(2)})_{\SO,\SOprime(G_{1}\times G_{2})}\\
\vspace{2mm}
& = \Vert g^{(1)} \otimes g^{(2)} \Vert_{2}^{-2} \int\limits_{\substack{G_{1}\times\ghat_{1} \times G_{2}\times \ghat_{2}}} \hspace{-0.5cm} (\kappa(A), \pi(\TFelement^{(1)}) g^{(1)}\otimes \pi(\TFelement^{(2)}) g^{(2)})_{\SO,\SOprime(G_{1}\times G_{2})} 
\\ & \hspace{4cm} \cdot \,  
(\overline{\pi(\TFelement^{(1)})g^{(1)} \otimes \pi(\TFelement^{(2)})g^{(2)}}, \sigma^{(1)}\otimes \sigma^{(2)})_{\SO,\SOprime(G_{1}\times G_{2})} \, d(\TFelement^{(1)},\TFelement^{(2)}) \\
& = \int\limits_{\substack{G_{1}\times\ghat_{1} \times G_{2}\times \ghat_{2}}} \hspace{-0.5cm} A\big( \pi(\TFelement^{(1)}) g^{(1)}, \pi(\TFelement^{(2)}) g^{(2)} \big) 
\\ & \hspace{2cm} \cdot \  
(\overline{\pi(\TFelement^{(1)})g^{(1)}},\sigma^{(1)})_{\SO,\SOprime(G_{1})} \ (\overline{\pi(\TFelement^{(2)})g^{(2)}}, \sigma^{(2)})_{\SO,\SOprime(G_{2})} \, d(\TFelement^{(1)},\TFelement^{(2)}).
\end{align*} 
\end{proof}

In the Introduction we stated that the spaces $\mathcal{A}(G_{1},G_{2})$ and $\mathcal{B}(G_{1},G_{2})$ are Banach spaces with respect to their subspace topologies which they naturally inherit from $\Bil(\SOprime(G_{1})\times\SOprime(G_{2}),\C)$ and $\Lin(\SOprime(G_{1}),\SO(G_{2}))$, respectively. At the same time it is clear that the induced norms themselves fail to capture the continuity requirements for operators in $\mathcal{A}(G_{1},G_{2})$ and $\mathcal{B}(G_{1},G_{2})$ as described in Definition \ref{def:the-spaces}. That is, the induced norm on $\mathcal{A}(G_{1},G_{2})$ can not distinguish between operators in $\Bil(\SOprime(G_{1})\times\SOprime(G_{2}))$ that belong to $\mathcal{A}(G_{1},G_{2})$ and those that do not. Similarly, the norm on $\Lin(\SOprime(G_{1}),\SO(G_{2}))$ can not detect if an operator actually belongs to $\mathcal{B}(G_{1},G_{2})$ or not. The results from Theorem \ref{th:1608a} show how we can define a norm on the spaces $\mathcal{A}(G_{1},G_{2})$ and $\mathcal{B}(G_{1},G_{2})$ that exactly captures operators with a kernel in $\SO(G_{1}\times G_{2})$.

\begin{corollary} \label{cor:2408a} Let $A$, $T$, $K$, $\kappa(T)$, $\kappa(A)$ be related as in Remark \ref{rem:2601a}. For $i=1,2$ fix a function $g^{(i)}\in\SO(G_{i})\backslash\{0\}$.
\begin{enumerate}
\item[(i)] $\ \ \Vert \cdot \Vert_{\mathcal{A}, g_1,g_2} : \mathcal{A}(G_{1},G_{2})\to \R_{0}^{+},$
\begin{align*} \Vert A \Vert_{\mathcal{A}, g_1,g_2} & = \int\limits_{\substack{G_{1}\times\ghat_{1} \times G_{2}\times\ghat_{2}}} \hspace{-0.5cm} \big\vert A\big(\pi(\TFelement^{(1)}) g^{(1)}, \pi(\TFelement^{(2)})g^{(2)}\big) \big\vert \, d(\TFelement^{(1)},\TFelement^{(2)})\\
& = \Vert K \Vert_{\SO(G_1 \times G_2),g_1 \otimes g_2} = \Vert \kappa(A) \Vert_{\SO(G_1 \times G_2),g_1 \otimes g_2}, \ A\in \mathcal{A}(G_{1},G_{2}),\end{align*}
defines a norm on $\mathcal{A}(G_{1},G_{2})$. This norm is equivalent to the subspace norm on $\mathcal{A}(G_{1},G_{2})$ induced by the space $\Bil(\SOprime(G_{1})\times \SOprime(G_{2}),\C)$. 
\item[(ii)] $ \ \ \Vert \cdot \Vert_{\mathcal{B}, g_1,g_2} :  \ \mathcal{B}(G_{1},G_{2}) \to \R_{0}^{+}$, 
\begin{align*}
 \Vert T \Vert_{\mathcal{B}, g_1,g_2} & = \int\limits_{\substack{G_{1}\times\ghat_{1} \times G_{2}\times\ghat_{2}}} \hspace{-0.5cm} \big\vert \big( \pi(\TFelement^{(2)}) g^{(2)}, T \pi(\TFelement^{(1)})g^{(1)}\big)_{\SO,\SOprime(G_{2})} \big\vert \, d(\TFelement^{(1)},\TFelement^{(2)}) \\
 & = \Vert K \Vert_{\SO(G_1 \times G_2),g_1 \otimes g_2} = \Vert \kappa(T) \Vert_{\SO(G_1 \times G_2),g_1 \otimes g_2}, \ T\in \mathcal{B}(G_{1},G_{2}),\end{align*}
defines a norm on $\mathcal{B}(G_{1},G_{2})$. This norm is equivalent to the subspace norm on $\mathcal{B}(G_{1},G_{2})$ induced by the space $\Lin(\SOprime(G_{1}),\SO(G_{2}))$.
\end{enumerate}
\end{corollary}
\begin{remark}Theorem \ref{th:1608a} implies that the integrals used to define the norms in Corollary \ref{cor:2408a}(i) and (ii) are finite \emph{exactly} when $A$ and $T$ belong to $\mathcal{A}(G_{1},G_{2})$ and $\mathcal{B}(G_{1},G_{2})$, respectively. 
\end{remark}

We can use \eqref{eq:integral-B} and Lemma \ref{le:STFT-SO-SOprime} with respect to the time-frequency plane $G_{2}\times\ghat_{2}$ to show that an operator in $\mathcal{B}(G_{1},G_{2})$ is uniquely determined by its action on all time-frequency shifts of a given function in $\SO(G_{1})$.

\begin{corollary} \label{cor:2408b} Fix $g\in \SO(G_{1})\backslash\{0\}$. An operator in $\mathcal{B}(G_{1},G_{2})$ is uniquely determined by its action on the set $\{ \pi(\TFelement)g \, : \, \TFelement\in G_{1}\times\ghat_{1}\}$.  Specifically, for all $T\in \mathcal{B}(G_{1},G_{2})$ and $\sigma^{(i)}\in\SOprime(G_{i})$, $i=1,2$,
\[ (T\sigma^{(1)},\sigma^{(2)})_{\SO,\SOprime(G_{2})} = \Vert g \Vert_{2}^{-2} \int_{G_{1}\times\ghat_{1}} \, \mathcal{V}_{g}\sigma^{(1)}(\TFelement) \, \cdot \, \big(T \, \pi(\TFelement) g, \sigma^{(2)}\big)_{\SO,\SOprime(G_{2})} \, d\TFelement.\]
\end{corollary}
We refer to Corollary \ref{cor:2408b} by saying that the following identity holds true in the weak sense: 
\begin{equation} \label{eq:1403c}  T \sigma = \Vert g \Vert_{2}^{-2} \int_{G_{1}\times\ghat_{1}} \mathcal{V}_{g}\sigma(\TFelement) \, \cdot \, T\,\pi(\TFelement) g \ d\TFelement \ \ \text{for all} \ \ \sigma\in\SOprime(G_{1}).\end{equation}

\subsubsection{A note on operator with kernel in $\SOprime$}
The results of Theorem \ref{th:1608a} and Corollaries \ref{cor:2408a} and \ref{cor:2408b} are not restricted to the operators in $\mathcal{A}(G_{1},G_{2})\cong\mathcal{B}(G_{1},G_{2})$, but can be formulated in a very similar form for the much larger spaces of operators that have a kernel in $\SOprime(G_{1}\times G_{2})$ (by use of the outer rather than the inner kernel theorem and Lemma \ref{le:STFT-SO-SOprime}). For operators in $\Lin(\SO(G_{1}),\SOprime(G_{2}))$ they take the following form.

\begin{proposition} \label{pr:1108} 
Given $f^{(i)},g^{(i)}\in \SO(G_{i})$, $g^{(i)}\neq 0$ for $i = 1,2$ one has for any operator  $ T  \in \Lin(\SO(G_{1}),\SOprime(G_{2}))$:  
\begin{enumerate} 
\item[(i)] $\ \ \Vert g^{(1)} \otimes g^{(2)} \Vert_{2}^{2} \, \cdot \,   (f^{(2)},Tf^{(1)})_{\SO,\SOprime(G_{2})} $
\[ = \int\limits_{\substack{G_{1}\times \ghat_{1} \times G_{2}\times \ghat_{2}}} \hspace{-0.5cm} \mathcal{V}_{g^{(1)}}f^{(1)}(\TFelement^{(1)}) \, \cdot\, \mathcal{V}_{g^{(2)}}f^{(2)}(\TFelement^{(2)}) \cdot \Big(\pi(\TFelement^{(2)}) g^{(2)}, T\circ \pi(\TFelement^{(1)}) g \Big)_{\SO,\SOprime(G_{2})} \, d(\TFelement^{(1)},\TFelement^{(2)}),\]
\item[(ii)] $ \displaystyle \ \ Tf^{(1)} = \Vert g^{(1)} \Vert_{2}^{-2} \int_{G_{1}\times \ghat_{1}} \mathcal{V}_{g^{(1)}} f^{(1)}(\TFelement^{(1)}) \,\cdot\, T \left( \pi(\TFelement^{(1)}) g \right ) \, d\TFelement^{(1)}$,
\item[(iii)] and $ \ \ \Vert \cdot \Vert_{\mathcal{B}', g_1,g_2} : \Lin(\SO(G_{1}),\SOprime(G_{2}))\to \R_{0}^{+}$,
\begin{align*}
     \Vert T \Vert_{\mathcal{B}', g_1,g_2} & = \sup_{ \substack{ \TFelement^{(1)}\in G_{1}\times\ghat_{1}\\ \TFelement^{(2)}\in G_{2}\times \ghat_{2} }} \vert (\pi(\TFelement^{(2)}) g^{(2)} , T \pi(\TFelement^{(1)}) g^{(1)} )_{\SO,\SOprime(G_{2})} \vert
     & = \Vert \kappa(T) \Vert_{\SOprime(G_1\times G_2),g_1\otimes g_2}  \end{align*}
defines a norm on $\Lin(\SO(G_{1}),\SOprime(G_{2}))$ which is equivalent to the usual operator norm on $\Lin(\SO(G_{1}),\SOprime(G_{2}))$.
\end{enumerate}
\end{proposition}

This result has the following consequence.

\begin{corollary} For $i=1,2$ take $g^{(i)}\in \SO(G_{i})\backslash\{0\}$. Every continuous and bounded function $F\in \mathbf{C}_{b}(G_{1}\times \ghat_{1}\times G_{2}\times\ghat_{2})$ defines a linear and bounded operator $T:\SO(G_{1})\to \SOprime(G_{2})$ via
\begin{equation} \label{eq:0802a} (f^{(2)}, T f^{(1)})_{\SO,\SOprime(G_{2})} = \int\limits_{\substack{G_{1}\times \ghat_{1}  \times G_{2}\times \ghat_{2}}} \hspace{-0.5cm}\mathcal{V}_{g^{(1)}} f^{(1)} (\TFelement^{(1)}) \, \cdot \, \mathcal{V}_{g^{(2)}} f^{(2)} (\TFelement^{(2)}) \, \cdot \, F(\TFelement^{(1)},\TFelement^{(2)}) \ d(\TFelement^{(1)},\TFelement^{(2)}). \end{equation}
Conversely, for every linear and bounded operator $T:\SO(G_{1})\to\SOprime(G_{2})$ there exists a (non-unique) function $F\in \mathbf{C}_{b}(G_{1}\times\ghat_{1}\times G_{2}\times \ghat_{2})$ (which also depends on $g^{(i)}$) such that \eqref{eq:0802a} holds. Moreover, the function $F$ can be taken to be in $\Lp{1}(G_{1}\times \ghat_{1}\times G_{2}\times \ghat_{2})$ if and only if $T$ has a kernel in $\SO(G_{1}\times G_{2})$. 
\end{corollary} 

\begin{remark}
This shows that it is possible to have a calculus for operators in $\Lin(\SO(G_{1}),\SOprime(G_{2}))$ where the operators are represented by bounded and continuous functions (rather than abstract functionals as in the outer kernel theorem). However, as the Corollary states, one has to pay for this by giving up on the non-uniqueness of such a representation.
\end{remark}

\subsubsection{A note on Gabor frames} \label{sec:gabor-frames}

Recall that a function $g\in\SO(G)$ generates a \emph{Gabor frame} for $\Lp{2}(G)$ with respect to a closed subgroup $\Lambda$ in $G\times\ghat$ (typically $\Lambda$ is a discrete and co-compact subgroup, a lattice, in the time-frequency plane) if there exist constants $A,B>0$ such that
\begin{equation} \label{eq:frame-ineq} A \, \Vert f \Vert_{2}^{2} \le \int_{\Lambda} \big\vert \big\langle f, \pi(\lambda) g \big\rangle \big\vert^{2} \, d\lambda \le B \, \Vert f \Vert_{2}^{2} \ \ \forall \ f\in \Lp{2}(G).\end{equation} 
In that case there exists a (not necessarily unique) function $h\in \SO(G)$ such that
\begin{equation} \label{eq:frame-rep} ( f, \sigma)_{\SO,\SOprime} = \int_{\Lambda} \langle f, \pi(\lambda) g \rangle \, ( \pi(\lambda) h , \sigma)_{\SO,\SOprime} \, d\lambda
\quad \text{for all} \quad f\in\SO(G), \ \sigma\in\SOprime(G). 
\end{equation}
We will not go into details of how pairs of function $g$ and $h$ can be found or be characterized so that \eqref{eq:frame-rep} holds. For general  Gabor and time-frequency analysis we refer to  \cite{ch16,gr01} and \cite{jale16-2}. 

We have already encountered a Gabor frame for $\Lp{2}(G)$ with respect to the subgroup $G\times\ghat$: Lemma \ref{le:STFT-SO-SOprime} shows that any non-zero function $g\in\SO(G)$ generates a Gabor frame for $\Lp{2}(G)$ with respect to $\Lambda = G\times\ghat$ (in \eqref{eq:1412a} take $\sigma$ to be induced by $\overline{f}$; since \eqref{eq:1412a} holds for all $f\in\SO(G)$, which is dense in $\Lp{2}(G)$, it follows that \eqref{eq:frame-ineq} is satisfied and that $A=B=\Vert g \Vert_{2}^{2}$). In this case, if $\tilde{g}$ is any other function in $\SO(G)$ such that $\langle g, \tilde{g}\rangle\ne 0$, then the pair $(g,h)$, where $h=(\langle\tilde{g} ,g\rangle)^{-1} \tilde{g}$ satisfies \eqref{eq:frame-rep}.

Using \eqref{eq:frame-rep} rather than \eqref{eq:1412a} for the proofs of Section \ref{sec:kernel-in-s0} leads to the following results for the operators in $\mathcal{B}(G_{1},G_{2})$ (we leave the formulation of the corresponding results for $\mathcal{A}(G_{1},G_{2})$ to the reader).

\begin{theorem} \label{th:frames} For $i=1,2$ let $g^{(i)}$ and $h^{(i)}$ be functions in $\SO(G_{i})$ such that they generate Gabor frames for $\Lp{2}(G_{i})$ with respect to a  subgroup $\Lambda_{i}$ in $G_{i}\times\ghat_{i}$ and such that \eqref{eq:frame-rep} holds. 
\begin{enumerate}
\item[(i)] $\ \ \Vert \cdot \Vert_{\mathcal{B},\Lambda_1\times\Lambda_2} :  \ \mathcal{B}(G_{1},G_{2}) \to \R_{0}^{+}$,
\[ \Vert T \Vert_{\Lambda_{1}\times\Lambda_{2}} = \int_{\Lambda_{1}\times\Lambda_{2}}  \big\vert \big( \pi(\lambda^{(2)}) g^{(2)}, T \pi(\lambda^{(1)})g^{(1)}\big)_{\SO,\SOprime(G_{2})} \big\vert \, d(\lambda^{(1)},\lambda^{(2)}), \ T\in \mathcal{B}(G_{1},G_{2}),\]
defines a norm on $\mathcal{B}(G_{1},G_{2})$. This norm is equivalent to the subspace norm on $\mathcal{B}(G_{1},G_{2})$ induced by the space $\Lin(\SOprime(G_{1}),\SO(G_{2}))$. For an operator $T\in \Lin\big(\SO(G_{1}),\SOprime(G_{2})\big)$ the norm is finite if and only if $T\in\mathcal{B}(G_{1},G_{2})$.
\item[(ii)] Given $T\in\Lin\big(\SO(G_{1}),\SOprime(G_{2})\big)$ with kernel $\kappa(T)\in \SO(G_{1}\times G_{2})$, then 
\begin{align*} & (T\sigma^{(1)},\sigma^{(2)})_{\SO,\SOprime(G_{2})} 
 \nonumber \\ 
 & = \int_{\Lambda_{1}\times\Lambda_{2}} \!\! \, \mathcal{V}_{g^{(1)}}\sigma^{(1)}(\lambda^{(1)}) \, \cdot \, \mathcal{V}_{g^{(2)}}\sigma^{(2)}(\lambda^{(2)})  \cdot  \big( \pi(\lambda^{(2)})h^{(2)}, T\pi(\lambda^{(1)})h^{(1)} \big)_{\SO,\SOprime(G_{2})} \ d(\lambda^{(1)},\lambda^{(2)}). \end{align*}
\item[(iii)] Any $T\in \mathcal{B}(G_{1},G_{2})$ satisfies 
\[ (T\sigma^{(1)},\sigma^{(2)})_{\SO,\SOprime(G_{2})} = \int_{\Lambda_{1}} \, \mathcal{V}_{g^{(1)}}\sigma^{(1)}(\lambda^{(1)}) \, \cdot \, (T \, \pi(\lambda^{(1)}) h^{(1)}, \sigma^{(2)})_{\SO,\SOprime(G_{2})} \, d\lambda^{(1)}.\]
\end{enumerate}
\end{theorem}

We can very explicit if $G=\R^{n}$. In this case it is known that for any $n\in\N$ the Gaussian function $g^{(n)}(x) = e^{-\pi x\cdot x}$, $x\in\R^{n}$ generates a Gabor frame for $\Lp{2}(\R^{n})$ with respect to the lattice $\Lambda = a\Z^{2n}\subset \R^{2n}$ whenever $0<a<1$. Hence in this case the integrals in Theorem \ref{th:frames} become a sum over lattice points. In particular, any linear and bounded operator $T$ from $\SO(\R^{n})$ into $\SOprime(\R^{m})$ has a  kernel in $\SO(\R^{n+m})$ if and only if the Gabor frame matrix of $T$ has coefficients in
$\ell^1(\Z^{2m})$, i.e.\ 
\[ \sum_{\substack{ \lambda^{(n)}\in a\Z^{2n} \\ \lambda^{(m)}\in a\Z^{2m}}} \big\vert \big( \pi(\lambda^{(m)}) g^{(m)} , T \pi(\lambda^{(n)}) g^{(n)} \big)_{\SO,\SOprime(\R^{m})}  \big\vert < \infty.\]

\subsection{Analogies with linear algebra} 
\label{sec:linear-albera}

If $A$ is an $n_{2}\times n_{1}$ matrix, then it defines an operator $\widehat{A}$ from $\C^{n_{1}}$ into $\C^{n_{2}}$,
\[ \widehat{A} : \C^{n_{1}} \to \C^{n_{2}}, \widehat{A} (v^{(1)}) = A \cdot v^{(1)}, \ \ v^{(1)}\in \C^{n_{1}}.\]
Conversely, if a linear operator $\widehat{A}$ from $\C^{n_{1}}$ into $\C^{n_{2}}$ is given and we use the standard basis for these spaces, then the matrix representation of $\widehat{A}$ is
\begin{equation} \label{eq:1312a} A(i,j) = (A(e_{j}))^{\top} \cdot e_{i}, \ \ i=1,\ldots,n_{2}, \ j=1,\ldots, n_{2} \end{equation}
and then $\widehat{A}(v^{(1)}) = A \cdot v^{(1)}$.
If a matrix $A$ is as above and if we let a matrix $B\in \C^{n_{3}\times n_{2}}$ define an operator $\widehat{B}$ from $\C^{n_{2}}$ into $\C^{n_{3}}$, then their composition, $\widehat{B}\circ\widehat{A}$,  is represented by the product of the two matrices. That is, 
\begin{equation} \label{eq:1312b} \widehat{B}\circ\widehat{A} : \C^{n_{1}} \to \C^{n_{3}}, \ \widehat{B}\circ \widehat{A} (v^{(1)}) = B\cdot A\cdot v^{(1)}, \ \ v^{(1)} \in \C^{n_{1}},\end{equation}
according to usual matrix multiplication: 
$$(B\cdot A)(i,j) = \displaystyle\sum_{k=1}^{n_{2}} B(i,k) \cdot A(k,j), \,\, i=1,\ldots, n_{3}, j=1,\ldots,n_{1}. $$

The next two results show that the inner kernel theorem allows us to extend both \eqref{eq:1312a} and \eqref{eq:1312b} from matrices to operators in $\mathcal{B}(G_{1},G_{2})$. In particular, the role of the unit vectors in $\C^{n}$ is taken by the Dirac delta distributions, $\delta_{x}:\SO(G)\to\C$, $\delta_{x}:f\mapsto f(x)$, $f\in\SO(G)$, $x\in G$. If $G_{i} = \Z/n_{i} \Z$, $i=1,2$, then the results reduce to the matrix representation of linear
mappings. 

\begin{lemma} \label{le:kernel-pointwise} Given $T\in \mathcal{B}(G_{1},G_{2})$ and $A\in\mathcal{A}(G_{1},G_{2})$  its kernel satisfies for  $x^{(i)}\in G_{i}, \, i=1,2$:
\begin{align*} 
& \kappa(T)(x^{(1)},x^{(2)}) = ( T\delta_{x^{(1)}}, \delta_{x^{(2)}})_{\SO,\SOprime(G_{2})} \ \ \text{and} \ \  \kappa(A)(x^{(1)},x^{(2)}) = A(\delta_{x^{(1)}},\delta_{x^{(2)}}),
\end{align*}
\end{lemma}
\begin{proof} 
It is easy to verify the equality $\delta_{x^{(1)},x^{(2)}}= \delta_{x^{(1)}}\otimes \delta_{x^{(2)}}$,
since 
\[ \big( f^{(1)}\otimes f^{(2)}, \delta_{x^{(1)}} \otimes \delta_{x^{(2)}}\big)_{\SO,\SOprime(G_{1}\times G_{2})} = f^{(1)}(x^{(1)}) \cdot f^{(2)}(x^{(2)}) 
=   (f^{(1)} \otimes f^{(2)})(x^{(1)}, x^{(2)}).\] 
The desired result now follows from the inner kernel theorem:
\begin{align*} & 
\kappa(T)(x^{(1)},x^{(2)}) = \big( \kappa(T) , \delta_{x^{(1)},x^{(2)}}\big)_{\SO,\SOprime(G_{1}\times G_{2})} \\
&
= \big( \kappa(T) , \delta_{x^{(1)}}\otimes \delta_{x^{(2)}}\big)_{\SO,\SOprime(G_{1}\times G_{2})} = \big( T \delta_{x^{(1)}} , \delta_{x^{(2)}} \big)_{\SO,\SOprime(G_{2})}.
\end{align*}
The equality for the kernel of $A$ follows in the same fashion.
\end{proof}
The role of the ``delta-basis'' in Lemma \ref{le:kernel-pointwise} can also be taken by a continuous Gabor frame (cf.\ Section \ref{sec:gabor-frames}). 
 \begin{corollary} For $i=1,2$ let $g^{(i)}\in\SO(G_{i})\backslash\{0\}$ and $x^{(i)}\in G_{i}$.
\begin{enumerate}
\item[(i)] For $T\in \mathcal{B}(G_{1},G_{2})$ one has  
\[ \kappa(T)(x^{(1)},x^{(2)}) = \Vert g^{(1)}\Vert_{2}^{-2} \int_{G_{1}\times\ghat_{1}} \overline{\big(\pi(\TFelement^{(1)}) g^{(1)}\big)(x^{(1)})} \cdot \big( T \circ \pi(\TFelement^{(1)}) g^{(1)}\big)(x^{(2)}) \, d\TFelement^{(1)}.\]
\item[(ii)] For $A\in\mathcal{A}(G_{1},G_{2})$ one has
\begin{align*} \kappa(A)(x^{(1)},x^{(2)}) = \Vert g^{(1)}\otimes g^{(2)} \Vert_{2}^{-2} & 
\int\limits_{\substack{G_{1}\times \ghat_{1} \times  G_{2}\times\ghat_{2}}} \overline{\big(\pi(\TFelement^{(1)}) g^{(1)}\big)(x^{(1)})} \, \overline{\big(\pi(\TFelement^{(2)}) g^{(2)}\big)(x^{(2)})} \\
& \hspace{1cm} \cdot \, A\big( \pi(\TFelement^{(1)}) g^{(1)}, \pi(\TFelement^{(2)}) g^{(2)} \big) \, d(\TFelement^{(1)},\TFelement^{(2)}).\end{align*}
\end{enumerate}
\end{corollary}
\begin{proof}
Combine equality \eqref{eq:integral-B} of Theorem \ref{th:1608a}, Corollary \ref{cor:2408b} with Lemma \ref{le:kernel-pointwise}.
\end{proof}
The composition rule of operators represented by matrices has the following analogous continuous formulation for operators 
with kernel in $\SO$.

\begin{lemma} \label{le:comp-of-operators}
If $T_{1}\in \mathcal{B}(G_{1},G_{2})$ and $T_{2}\in \mathcal{B}(G_{2},G_{3})$, then $T_{2}\circ T_{1} \in \mathcal{B}(G_{1},G_{3})$ and
\[ \kappa(T_{2}\circ T_{1})(x^{(1)},x^{(3)})  = \int_{G_{2}} \kappa(T_{1})(x^{(1)},x^{(2)}) \, \cdot \, \kappa(T_{2})(x^{(2)},x^{(3)}) \, dx^{(2)}, \ \ x^{(i)} \in G_{i}, \ i=1,3.\]
Moreover, using the norm on $\mathcal{B}$ as defined in Corollary \ref{cor:2408a}, there exists a constant $c>0$ such that
\[ \Vert T_{2}\circ T_{1} \Vert_{\mathcal{B}} \le c \, \Vert T_{2} \Vert_{\mathcal{B}} \, \Vert T_{1} \Vert_{\mathcal{B}}.\]
\end{lemma}
\begin{corollary} The Banach space $(\mathcal{B}(G,G),\Vert\cdot\Vert_{\mathcal{B}})$ forms a Banach algebra under composition.
\end{corollary}
\begin{proof}[Proof of Lemma \ref{le:comp-of-operators}]
Let us first show that the integral is well-defined. 
By Lemma \ref{le:SO-tensor-factorization} we can write
\[ \kappa(T_{1}) = \sum_{j\in \N} f_{j}^{(1)} \otimes f_{j}^{(2)}  \ \ \text{and} \ \ \kappa(T_{2}) = \sum_{j\in \N} h_{j}^{(2)} \otimes h_{j}^{(3)}\]
for suitable $f_{j}^{(i)}\in \SO(G_{i})$, $i=1,2$ and $h_{j}^{(i)}\in \SO(G_{i})$, $i=2,3$ and where $j\in \N$. Furthermore,
\[ \sum_{j\in \N} \Vert f_{j}^{(1)}\Vert_{\SO} \, \Vert f_{j}^{(2)} \Vert_{\SO} < \infty \ \ \text{and} \ \ \sum_{j\in \N} \Vert h_{j}^{(1)}\Vert_{\SO} \, \Vert h_{j}^{(2)} \Vert_{\SO} < \infty. \]
Because $\SO(G)$ is continuously embedded into $\Lp{2}(G)$ and into $\Lp{\infty}(G)$  the following estimate applies: for all $x^{(i)} \in G_{i}$, $i=1,2,3$ 
\begin{align*} 
& \quad \int_{G_{2}} \big\vert \kappa(T_{1})(x^{(1)},x^{(2)}) \, \cdot \, \kappa(T_{2})(x^{(2)},x^{(3)}) \big\vert \, dx^{(2)}\\
& \le \int_{G_{2}} \sum_{j,k\in\N} \vert f_{j}^{(1)}(x^{(1)}) \cdot f_{j}^{(2)}(x^{(2)}) \cdot h_{k}^{(2)}(x^{(2)}) \cdot h_{k}^{(3)}(x^{(3)}) \vert \, dx^{(2)} \\
& \le \sum_{j,k\in\N} \Vert f_{j}^{(1)}\Vert_{\infty} \, \Vert h_{k}^{(3)} \Vert_{\infty} \, \int_{G_{2}} \vert f_{j}^{(2)}(x^{(2)}) \cdot h_{k}^{(2)}(x^{(2)}) \vert \, dx^{(2)} \\
& \le \sum_{j,k\in\N} \Vert f_{j}^{(1)}\Vert_{\infty} \, \Vert h_{k}^{(3)} \Vert_{\infty} \,  \Vert f_{j}^{(2)} \Vert_{2} \, \Vert h_{k}^{(2)} \Vert_{2} \\
& \le c \, \sum_{j,k\in\N} \Vert f_{j}^{(1)}\Vert_{\SO} \, \Vert h_{k}^{(3)} \Vert_{\SO} \,  \Vert f_{j}^{(2)} \Vert_{\SO} \, \Vert h_{k}^{(2)} \Vert_{\SO} < \infty,
\end{align*}
for some $c>0$. This shows that the integral and thus the function $\kappa(T_{2}\circ T_{1}):G_{1}\times G_{3}\to \C$ is well-defined. 
Note that
\begin{align*} 
\kappa(T_{2}\circ T_{1})(x^{(1)},x^{(3)}) & = \int_{G_{2}} \kappa(T_{1})(x^{(1)},x^{(2)}) \, \kappa(T_{2})(x^{(2)},x^{(3)}) \, dx^{(2)} \\
& = \int_{G_{2}} \sum_{j,k\in\N} f_{j}^{(1)}(x^{(1)}) \, f_{j}^{(2)}(x^{(2)}) \, h_{k}^{(2)}(x^{(2)}) \, h_{k}^{(3)}(x^{(3)})  \, dx^{(2)} \\
& = \sum_{j,k\in\N} (f_{j}^{(2)},h_{k}^{(2)})_{\SO,\SOprime(G_{2})} \, \big( f_{j}^{(1)} \otimes h_{k}^{(3)}\big)(x^{(1)},x^{(3)} ).
\end{align*}
Hence $\kappa(T_{1}\circ T_{2}) = \sum_{j,k\in\N} (f_{j}^{(2)},h_{k}^{(2)})_{\SO,\SOprime(G_{2})} \, f_{j}^{(1)} \otimes h_{k}^{(3)}$.
The above calculation shows that
\[ \sum_{j,k\in\N} \Vert (f_{j}^{(2)}, h_{k}^{(2)}) \, f_{j}^{(1)} \Vert_{\SO} \, \Vert h_{k}^{(3)} \Vert_{\SO} \le \infty. \]
Hence $\kappa(T_{2}\circ T_{1}) \in \SO(G_{1})\hat{\otimes} \SO(G_{3})$. By Lemma \ref{le:SO-tensor-factorization} this implies that $\kappa(T_{2}\circ T_{1})\in \SO(G_{1}\times G_{3})$ as well as the moreover-part of the lemma. Let us show that the function which we defined as $\kappa(T_{2}\circ T_{1})$ indeed is the kernel of the operator $T_{2}\circ T_{1}$: if $\sigma^{(i)}\in \SOprime(G_{i})$, $i=1,3$,
then
\begin{align*}
& \quad ( T_{2}\circ T_{1} \sigma^{(1)},\sigma^{(3)})_{\SO,\SOprime(G_{3})} = (\kappa(T_{2}) , T_{1}\sigma^{(1)} \otimes \sigma^{(3)})_{\SO,\SOprime(G_{2}\times G_{3})} \\
& = \sum_{k\in\N} ( h_{k}^{(2)} \otimes h_{k}^{(3)}, T_{1}\sigma^{(1)} \otimes \sigma^{(3)})_{\SO,\SOprime(G_{2}\times G_{3})} \\
& = \sum_{k\in\N} (T_{1}\sigma^{(1)},h_{k}^{(2)})_{\SO,\SOprime(G_{2})} \, (h_{k}^{(3)}, \sigma^{(3)})_{\SO,\SOprime(G_{3})} \\
& = \sum_{k\in\N} (\kappa(T_{1}), \sigma^{(1)}\otimes h_{k}^{(2)})_{\SO,\SOprime(G_{1}\times G_{2})} \, (h_{k}^{(3)}, \sigma^{(3)})_{\SO,\SOprime(G_{3})} \\
& = \sum_{j,k\in\N} (f_{j}^{(1)}\otimes f_{j}^{(2)}, \sigma^{(1)}\otimes h_{k}^{(2)})_{\SO,\SOprime(G_{1}\times G_{2})} \, (h_{k}^{(3)}, \sigma^{(3)})_{\SO,\SOprime(G_{3})} \\
& = \Big( \sum_{j,k\in\N} (f_{j}^{(2)},h_{k}^{(2)})_{\SO,\SOprime(G_{2})} \, f_{j}^{(1)} \otimes h_{k}^{(3)} , \sigma^{(1)} \otimes \sigma^{(3)} \Big)_{\SO,\SOprime(G_{1}\times G_{3})}.
\end{align*}

\end{proof}

\subsection{Some examples, nuclearity and trace-class results}
\label{sec:examples}
\begin{example} \label{ex:rank-one-operator}

The prototypical example of an element in $\SO(G_{1}\times G_{2})$ is the tensor-product function $f = f^{(1)} \otimes f^{(2)}$, $f^{(i)}\in \SO(G_{i})$, $i=1,2$. It is not difficult to show that the unique corresponding operators $A$, $T$ and $S$ 
according to Theorem \ref{th:new-inner-kernel-theorem} are the following ones:
\begin{align*} 
& A : \SOprime(G_{1})\times \SOprime(G_{2}) \to \C, \ A(\sigma^{(1)},\sigma^{(2)}) = (f^{(1)},\sigma^{(1)})_{\SO,\SOprime(G_{1})} \, (f^{(2)},\sigma^{(2)})_{\SO,\SOprime(G_{2})},\\
& T : \SOprime(G_{1})\to \SO(G_{2}), \ T(\sigma^{(1)}) = (f^{(1)},\sigma^{(1)})_{\SO,\SOprime(G_{1})} \cdot f^{(2)},\\
& S : \SOprime(G_{2})\to \SO(G_{1}), \ S(\sigma^{(2)}) = (f^{(2)},\sigma^{(2)})_{\SO,\SOprime(G_{2})} \cdot f^{(1)},
\end{align*}
where $\sigma^{(i)}\in \SOprime(G_{i})$, $i=1,2$. \end{example} 
Observe that the range of the operators $T$ and $S$ in Example \ref{ex:rank-one-operator} is one-dimensional. Naturally, not all rank-one operators have a kernel in $\SO$. For example, let $f^{(1)}$ be a function in $\SO(G_{1})$ and let $f^{(2)}$ be a function in $\Lp{2}(G_{2})$ which is not also in $\SO(G_{2})$, then 
\[ T(\sigma^{(1)}) = (f^{(1)},\sigma^{(1)})_{\SO,\SOprime(G_{1})} \cdot f^{(2)}\]
is a bounded rank-one operator from $\SOprime(G_{1})$ into $\Lp{2}(G_{2})$ which does not have a kernel in $\SO(G_{1}\times G_{2})$. Indeed, if we restrict $T$ to an operator from $\SO(G_{1})$ into $\Lp{2}(G_{2})\subseteq \SOprime(G_{2})$, then (by the outer kernel theorem) its kernel in $\SOprime(G_{1}\times G_{2})$ is the functional induced by the function $f^{1}\otimes f^{2}$. Similarly, one can show that the operator 
\[ S(h^{(2)}) = \langle f^{(2)},\overline{h^{(2)}}\rangle_{\Lp{2}(G_{2})} \cdot f^{(1)}, \ \ h^{(2)}\in \Lp{2}(G_{2}), \]
is a linear and bounded rank-one operator from $\Lp{2}(G_{2})$ into $\SO(G_{1})$ with kernel $f^{2}\otimes f^{1}$.

\begin{remark} \label{rem:tensor-factorization} By Lemma \ref{le:SO-tensor-factorization}   every  $f \in \SO(G_{1}\times G_{2})$ has a representation 
\[f = \sum_{j\in \N} f^{(1)}_{j} \otimes f^{(2)}_{j} 
\ \ \text{such that} \ \ \sum_{j\in\N} \Vert f^{(1)}_{j} \Vert_{\SO} \, \Vert f^{(2)}_{j}\Vert_{\SO} < \infty. \]
The inner kernel theorem  implies that the corresponding operator $T\in\mathcal{B}(G_{1},G_{2})$ satisfies
\[ T \sigma^{(1)} = \sum_{j\in \N} (f_{j}^{(1)} , \sigma^{(1)} )_{\SO,\SOprime(G_{1})} \cdot  f_{j}^{(2)} \ \ \text{for all} \ \ \sigma^{(1)}\in\SOprime(G_{1}),\]
where the sum is absolutely convergent in the norm on $\mathcal{B}$.
\end{remark}
This immediately leads to the following.
\begin{corollary} \label{cor:0201} Finite-rank operators in $\mathcal{B}(G_{1},G_{2})$ are norm dense in $\mathcal{B}(G_{1},G_{2})$. In fact, for any $T\in \mathcal{B}(G_{1},G_{2})$ there exists a series of rank one operators $(T_{n})_{n\in\N}$ in $\mathcal{B}(G_{1},G_{2})$ such that
\[ \big\Vert T - \sum_{n\in\N} T_{n} \big\Vert_{\mathcal{B}} = 0 \ \ \text{and} \ \ \sum_{n\in\N} \Vert T_{n} \Vert_{\mathcal{B}} < \infty.\]
\end{corollary}

Because $\SO$ is dense in $\Lp{2}$ it follows that the finite-rank operators of $\mathcal{B}(G_{1},G_{2})$ are dense in the space of Hilbert-Schmidt operators from $\Lp{2}(G_{1})$ into $\Lp{2}(G_{2})$.

\begin{corollary} \label{cor:0301} All operators in $\mathcal{B}(G_{1},G_{2})$ are nuclear operators from the Banach space $\SOprime(G_{1})$ into the Banach space $\SO(G_{2})$.
\end{corollary}
\begin{proof} By \cite[Chapter III, \S 7]{scwo99} all nuclear operators from the Banach space $\SOprime(G_{1})$ into the Banach space $\SO(G_{2})$ are of the form 
\[ T:\SOprime(G_{1})\to \SO(G_{2}), \ T \sigma = \sum_{j\in \N} \psi^{(1)}_{j}(\sigma^{(1)}) \cdot f^{(2)}_{j},\]
where $(\psi_{j}^{(1)})$ is a sequence in $\SOdoubleprime(G_{1})$ and $(f^{(2)}_{j})$ is a sequence in $\SO(G_{2})$ such that 
\[ \sum_{j\in \N} \Vert \psi_{j}^{(1)} \Vert_{\SOdoubleprime} \, \Vert f^{(2)}_{j} \Vert_{\SO} < \infty.\]
Remark \ref{rem:tensor-factorization} combined with the fact that $\SO(G_{1})$ is continuously embedded into $\SOdoubleprime(G_{1})$ via the natural embedding implies that all operators in $\mathcal{B}(G_{1},G_{2})$ are nuclear. 
\end{proof}
By the embedding of $\SO$ into $\SOprime$ as described in Lemma \ref{le:SOprime-induced-by-SO} it follows that all the operators in $\mathcal{B}(G_{1},G_{2})$ are also nuclear operators from $\SO(G_{1})$ into $\SO(G_{2})$; from $\SO(G_{1})$ into $\SOprime(G_{2})$; and from $\SOprime(G_{1})$ into $\SOprime(G_{2})$. 

\begin{corollary} \label{cor:0401} An operator $T$ in $\mathcal{B}(G,G)$ with kernel $\kappa(T)\in \SO(G\times G)$ is a trace-class operator on both $\SO(G)$ and $\SOprime(G)$. Its trace satisfies $\textnormal{tr}(T) = \int_{G} \kappa(T)(x,x) \,dx$. 
\end{corollary} 
\begin{proof}
Following  \cite{ru51} we say that an operator $T$ on a Banach space $\Bsp$ is of
trace class if it has the form
\[ T: \Bsp \to \Bsp, \ Tx = \sum_{j\in\N} ( x, \sigma_{j})_{\Bsp,\Bsp'} \cdot b_{j} \ \ \text{for all} \ \ x\in \Bsp,\]
for  suitable sequences $(\sigma_{j})$ in $\Bsp'$ and $(b_{j})$ in $\Bsp$ such that $\sum_{j} \Vert \sigma_{j} \Vert_{\Bsp'} \, \Vert b_{j} \Vert_{\Bsp} < \infty$.
The trace of $T$ is $\textnormal{tr}(T) = \sum_{j\in \N} (b_{j},\sigma_{j})_{\Bsp,\Bsp'}$. Remark \ref{rem:tensor-factorization} shows that for $\Bsp=\SO(G)$ or $\Bsp=\SOprime(G)$ the operators in $\mathcal{B}(G,G)$ have the desired form. Here we also use that $\SO$ is continuously embedded into $\SOprime$ and $\SOdoubleprime$ via Lemma \ref{le:SOprime-induced-by-SO} and the canonical embedding $\iota:\SO\to\SOdoubleprime$, respectively. Before we show that $\textnormal{tr}(T) = \int_{G} \kappa(T)(x,x) \, dx$ we will first prove that the integral is well-defined. The inner kernel theorem states that $\kappa(T)$ is a function in $\SO(G\times G)$. Observe that the diagonal $\{ (x,y)\in G\times G: x=y\}$ is a closed subgroup of $G\times G$. It is a fact (see \cite[Theorem 7]{fe81-2} or \cite[Theorem 5.7]{ja19}) that the restriction of an $\SO$ function to a closed subgroup belongs to $\SO$ of that subgroup. In our case this means that $x\mapsto\kappa(T)(x,x)$ is a function in $\SO(G)$ and, in particular, that it is integrable (as $\SO\subseteq \Lp{1}$). Now, since $\kappa(T) = \sum_{j\in\N} f_{j}^{(1)}\otimes f_{j}^{(2)}$,
\begin{align*}
\textnormal{tr}(T) & = \sum_{j\in\N} (f_{j}^{(1)},f_{j}^{(2)})_{\SO,\SOprime(G)} = \sum_{j\in\N} \int_{G} f_{j}^{(1)}(x) \, f_{j}^{(2)}(x) \, dx \\
& = \int_{G} \sum_{j\in\N} f_{j}^{(1)}(x) \, f_{j}^{(2)}(x) \, dx = \int_{G} \kappa(T)(x,x) \, dx.
\end{align*}
\end{proof}
\begin{remark} Since $\SO$ is continuously embedded into $\SOprime$ and $(\SOprime,w^{*})'\cong \SO$ it is reasonable to extend the definition of trace-class operator from \cite{ru51} used in the proof of Corollary \ref{cor:0401} as follows: We say a (linear and continuous) operator from $\SOprime(G)$ with the \ws \ topology into $\SO(G)$ with its norm topology is of trace-class if 
\[ T : \SOprime(G) \to \SO(G), \ T \sigma = \sum_{j\in\N} (f^{(1)}_{j}, \sigma)_{\SO,\SOprime} \cdot f^{(2)}_{j} \ \ \text{for all} \ \ \sigma\in\SOprime(G),\]
for suitable sequences $(f^{(i)}_{j})$ in $\SO(G)$ and such that $\sum_{j} \Vert f^{(1)}_{j} \Vert \, \Vert f^{(2)}_{j} \Vert < \infty$. The trace of such an operator is then $\textnormal{tr}(T) = \sum_{j\in\N} (f^{(1)}_{j},f^{(2)}_{j})_{\SO,\SOprime}$. In that case, it is clear from Remark \ref{rem:tensor-factorization} that $\mathcal{B}(G,G)$ coincides \emph{exactly} with the  trace-class operators defined in this way. 
\end{remark}

Let us consider another important example of elements in $\mathcal{B}(G,G)$.
\begin{example}[Product-convolution operators] \label{ex:prod-conv-op} For any two functions $h_{1}$ and $h_{2}$ in $\SO(G)$  the product-convolution operator
\[ PC_{h_{1},h_{2}} : \SOprime(G) \to \Misp(G)\cong\SO(G), \ PC_{h_{1},h_{2}}(\sigma) = (\sigma\cdot h_{1})*h_{2},\]
and the convolution-product operator
\[ CP_{h_{1},h_{2}} : \SOprime(G) \to \Misp(G)\cong\SO(G), \ CP_{h_{1},h_{2}}(\sigma) = (\sigma* h_{1})\cdot h_{2},\]
are linear and bounded operators, which send norm bounded \ws \ convergent nets in $\SOprime(G)$ into norm convergent nets in $\Misp(G)\cong\SO(G)$. That is, both operators belong to $\mathcal{B}(G,G)$.
One can show that 
$\kappa(PC_{h_{1},h_{2}})= \tau_{1}(h_{1}\otimes h_{2})$ and  $\kappa(CP_{h_{1},h_{2}})= \tau_{2}(h_{1}\otimes h_{2})$, where
\begin{align*}
& \tau_{1}:\SO(G\times G)\to \SO(G\times G), \ \tau_{1}(f)(s,t) = f(s,t-s), \\
& \tau_{2}:\SO(G\times G)\to \SO(G\times G), \ \tau_{2}(f)(s,t) = f(t-s,t). 
\end{align*}
\end{example}

Product-convolution operators can be used to prove Lemma \ref{le:bounded-S0prime-approx} (see \cite[Proposition 6.15]{ja19} for the details). The kernel theorems translate Lemma \ref{le:bounded-S0prime-approx} into a statement for operators:
\begin{lemma} \label{le:0502a} For any operator $T\in\Lin(\SO(G_{1}),\SOprime(G_{2}))$ there exists a  net of operators $(T_{\netparam})$ in $\mathcal{B}(G_{1},G_{2})$, bounded in $\Lin(\SO(G_{1}),\SOprime(G_{2}))$ such that, for all $f^{(i)}\in \SO(G_{i})$, $i=1,2$, 
\[ \lim_{\netparam} \big\vert \big( f^{(2)}, (T - T_{\netparam}) f^{(1)} \big)_{\SO,\SOprime(G_{2})} \big\vert = 0, \ \ \Vert T_{\netparam} \Vert_{\textnormal{op},\SO\to\SOprime} \le \Vert T \Vert_{\textnormal{op},\SO\to\SOprime}.\]
\end{lemma}

Similar to Lemma \ref{le:0502a}, the inner kernel theorem can be used to translate Lemma \ref{le:S0-Gabor-expansion} from a statement of $\SO$ to a statement of $\mathcal{B}(G_{1},G_{2})$.

\begin{proposition} \label{pr:operator-gabor-rep} Let $T_{0}$ be a non-trivial operator in $\mathcal{B}(G_{1},G_{2})$. The operators $T$ in $\mathcal{B}(G_{1},G_{2})$ are exactly
those of the form 
\begin{align*} T  = \sum_{j\in \N} c_{j} \, \pi(\TFelement^{(2)}_{j} ) \circ T_{0} \circ \pi(\TFelement^{(1)}_{j} ), \end{align*}
where $c\in \ell^{1}(\N)$ and $(\TFelement^{(i)}_{j})$ are sequences in $G_{i}\times\ghat_{i}$, $i=1,2$. The sum converges in $\mathcal{B}(G_{1},G_{2})$ and furthermore $\Vert T \Vert = \inf \Vert c \Vert_{1} $, where the infimum is taken over all admissible representations of $T$ as above, defines an equivalent norm on $\mathcal{B}(G_{1},G_{2})$.
\end{proposition}

A similar statement is true for the space $\mathcal{A}(G_{1},G_{2})$. In that case, if $A_{0}$ is a non-trivial element in $\mathcal{A}(G_{1},G_{2})$, then all operators in $\mathcal{A}(G_{1},G_{2})$ are exactly those of the form
\[ A(\sigma^{(1)},\sigma^{(2)}) = \sum_{j\in \N} c_{j} \, A_{0} \big(\pi(\TFelement^{(1)}_{j}) \sigma^{(1)}, \pi(\TFelement^{(2)}_{j}) \sigma^{(2)} \big)\]
with $c$ and $\TFelement_{j}^{(i)}$ in Proposition \ref{pr:operator-gabor-rep}.

\begin{proof}[Proof of Proposition \ref{pr:operator-gabor-rep}] By Lemma \ref{le:S0-Gabor-expansion} we know that for any $T\in \mathcal{B}(G_{1},G_{2})$ there exists a sequence $c\in \ell^{1}(\N)$ and a sequence $(x^{(1)}_{j},x^{(2)}_{j},\omega^{(1)}_{j},\omega^{(2)}_{j})$ in $G_{1}\times G_{2}\times \ghat_{1}\times\ghat_{2}$ such that 
\[ \kappa(T) = \sum_{j\in\N} c_{j} E_{\omega^{(1)}_{j},\omega^{(2)}_{j}} T_{x^{(1)}_{j},x^{(2)}_{j}} \, \kappa(T_{0}).\]
Hence
\begin{align*}  (T\sigma^{(1)},\sigma^{(2)})_{\SO,\SOprime(G_{2})} & = (\kappa(T), \sigma^{(1)}\otimes \sigma^{(2)})_{\SO,\SOprime(G_{1}\times G_{2})} \\
& = \sum_{j\in\N} c_{j} \, \big( E_{\omega^{(1)}_{j},\omega^{(2)}_{j}} T_{x^{(1)}_{j},x^{(2)}_{j}} \kappa(T_{0}),\sigma^{(1)}\otimes \sigma^{(2)}\big)_{\SO,\SOprime(G_{1}\times G_{2})} \\
& = \sum_{j\in\N} c_{j} \, \big( \kappa(T_{0}), [T_{-x^{(1)}_{j}} E_{\omega^{(1)}_{j}} \sigma^{(1)}] \otimes [T_{-x^{(2)}_{j}} E_{\omega^{(2)}_{j}} \sigma^{(2)}] \big)_{\SO,\SOprime(G_{1},G_{2})} \\
& = \sum_{j\in\N} c_{j} \, \big( T_{0}\circ \, T_{-x^{(1)}_{j}} \, E_{\omega^{(1)}_{j}} \sigma^{(1)}, T_{-x^{(2)}_{j}} E_{\omega^{(2)}_{j}} \sigma^{(2)} \big)_{\SO,\SOprime(G_{2})} \\
& = \sum_{j\in\N} \underbrace{\omega^{(1)}_{j}(x^{(1)}_{j}) \, c_{j}}_{\in\ell^{1}(\N)} \, \big( \underbrace{E_{\omega^{(2)}_{j}} \, T_{x^{(2)}_{j}}}_{=:\pi(\TFelement^{(2)}_{j})} \circ T_{0} \circ  \underbrace{E_{\omega^{(1)}_{j}} \, T_{-x^{(1)}_{j}}}_{=:\pi(\TFelement^{(1)}_{j})} \sigma^{(1)},   \sigma^{(2)} \big)_{\SO,\SOprime(G_{2})}.
\end{align*}
\end{proof}

If we take $T_{0}$ to be a rank-one operator in $\mathcal{B}(G_{1},G_{2})$ as in Example \ref{ex:rank-one-operator} with $\kappa(T_{0})=f^{(1)}\otimes f^{(2)}$, then Proposition \ref{pr:operator-gabor-rep} states that any operator $T\in\mathcal{B}(G_{1},G_{2})$ has the form
\[  T \sigma^{(1)}  = \sum_{j\in\N} c_{j} \big( \pi(\TFelement^{(1)}_{j}) f^{(1)} , \sigma^{(1)} \big)_{\SO,\SOprime(G_{1})} \, \pi(\TFelement^{(2)}_{j}) f^{(2)} \]
for all $\sigma^{(1)}\in\SOprime(G_{1})$,
for a suitable sequence $c\in\ell^{1}(\N)$ and sequences $(\TFelement^{i}_{j})$ in $G_{i}\times\ghat_{i}$, $i=1,2$.

\subsection{Regularizing approximations of the identity}
\label{sec:reg-app-id}

Since $\SO$ is \ws \ dense in $\SOprime$ it is possible to approximate the kernel $\kappa(T)\in\SOprime(G_{1}\times G_{2})$ of a general operator $T$ in $\Lin(\SO(G_{1}),\SOprime(G_{2}))$ by a net (or sequence) of functions $\kappa_{\netparam}$ in $\SO(G_{1}\times G_{2})$ that converges in the \ws  sense towards $\kappa(T)$. The  associated operators $T_{\netparam}\in\mathcal{B}(G_{1},G_{2})$ satisfy
\begin{equation} \lim_{\netparam} \big\vert \big( f^{(2)} , (T - T_{\netparam}) f^{(1)}\big)_{\SO,\SOprime(G_{2})} \big\vert = 0 \ \ \text{for all} \ \ f^{(i)}\in\SO(G_{i}), \ i=1,2. \label{eq:0702a} \end{equation}
We saw this already in Lemma \ref{le:0502a}.
In this section we propose a construction of a net of operators $(T_{\netparam})$ in $\mathcal{B}(G_{1},G_{2})$ such that \eqref{eq:0702a} holds, that is \emph{not} based on the modification of the kernel per se (which is the idea behind Lemma \ref{le:0502a}), but rather by a composition of the given operator $T$ with certain operators:
 we introduce the idea of a  \emph{regularizing approximations of the identity}.

\begin{definition} \label{def:reg-app-id} A \emph{regularizing approximation of the identity} of $\SO(G)$ is a net of operators $(T_{\netparam})$  in $\mathcal{B}(G,G)$ (equivalently  $\kappa(T_{\netparam})\in \SO(G\times G)$) for each $\netparam$ and which satisfies the following conditions:

\begin{enumerate}[label=(\roman*)]
\item \label{strong-id} $\lim_{\netparam} \Vert T_{\netparam} f - f \Vert_{\SO} = 0 \ \ \text{for all} \ \ f\in\SO(G)$,
\item \label{SO-norm-bound} $\sup_{\netparam} \Vert T_{\netparam} \Vert_{\textnormal{op},\SO\to\SO} < \infty$,
\item \label{SOprime-norm-bound} $\sup_{\netparam} \Vert T_{\netparam} \Vert_{\textnormal{op},\SOprime\to\SOprime} < \infty$,

\item \label{weak-id} $\lim_{\netparam} \vert (f,T_{\netparam} \sigma - \sigma )_{\SO,\SOprime(G)} \vert = 0 \ \ \text{for all} \ \ f\in\SO(G), \ \sigma\in\SOprime(G)$.
\end{enumerate} 
\end{definition}
\begin{remark} Statement (i) and (ii) for the adjoint operators $(T_{\netparam}^{\times})$ implies (iv) and (iii), respectively. Hence we need only conditions (i) and (ii) for self-adjoint operators. Moreover, (i) implies (ii) for the case of a sequence of operators, due to the Banach-Steinhaus principle. 
\end{remark}

We list three examples of such families of operators at the end of this section.
It is straightforward to show that the properties of a regularizing approximation of the identity implies convergence of $\Lp{2}$.
\begin{lemma} If $(T_{\netparam})$ is a regularizing approximation of the identity for $\SO(G)$, then 
\[ \sup_{\netparam} \Vert T_{\netparam} \Vert_{\textnormal{op},\Lp{2}\to \Lp{2}} < \infty \ \ \text{and} \ \ \lim_{\netparam} \Vert T_{\netparam} f - f\Vert_{2} = 0 \ \ \text{for all} \ \ f\in \Lp{2}(G).\]
Moreover, the net $(\kappa(T_{\netparam}))$ in $\SO(G\times G)\subseteq\SOprime(G\times G)$ converges towards the kernel of the identity operator in the \ws sense.
\end{lemma}
\begin{proof}
The first statement follows by interpolation theory for operators and assumptions \ref{SO-norm-bound} and \ref{SOprime-norm-bound} in Definition \ref{def:reg-app-id}. Now, since $\SO$ is continuously embedded and dense in $\Lp{2}$, Definition \ref{def:reg-app-id}\ref{strong-id} implies that
\[ \lim_{\netparam} \Vert T_{\netparam} f - f \Vert_{2} \to 0 \ \ \text{for all} \ \ f\in \Lp{2}(G).\]
The moreover part follows from the fact that Definition \ref{def:reg-app-id}\ref{weak-id} implies \eqref{eq:0702a}
\end{proof}

Regularizing approximations of the identity allow us to construct a concrete family of operators that have kernels in $\SO(G_{1}\times G_{2})$, which approximate in the \ws \ sense any given operator that has as (abstract) kernel in $\SOprime(G_{1}\times G_{2})$.

\begin{proposition} \label{pr:reg-of-id}
For $i=1,2$ let $(T_{\netparam}^{(i)})$ be a regularizing approximation of the identity for $\SO(G_{i})$. For any operator $T \in \Lin(\SO(G_{1}),\SOprime(G_{2}))$ the collection of operators $(T_{\netparam})$,  $T_{\netparam} := T_{\netparam}^{(2)} \circ T \circ T_{\netparam}^{(1)}$ is such that
\begin{enumerate}
\item[(i)] $T_\netparam\in \mathcal{B}(G_{1},G_{2})$ for each $\netparam$, i.e., $\kappa(T_{\netparam})\in\SO(G_{1}\times G_{2})$,
\item[(ii)] $\lim_{\netparam} \big\vert \big( f^{(2)}, (T - 
T_{\netparam})f^{(1)}\big)_{\SO,\SOprime(G_{2})} \big\vert = 0$ for all $f^{(i)}\in\SO(G_{i})$, $i=1,2$,
\item[(iii)] $\kappa(T_{\netparam})$ converges to $\kappa(T)$ in the \ws \ sense.
\item[(iv)] $\sup_{\netparam} \Vert T_{\netparam} \Vert_{\textnormal{op},\SO\to\SOprime} < \infty$. 

\item[(v)]For $T \in  \Lin(\Lp{2}(G_{1}),\Lp{2}(G_{2})) \text{ one has }
 \lim_{\netparam} \Vert (T - T_{\netparam}) f \Vert_{\Lp{2}(G_{2})} = 0 \ \ \text{for all} \ \  f\in\Lp{2}(G_{1}).$
 
\item[(vi)]For $T \in \Lin(\SO(G_{1}),\SO(G_{2})) \text{ one has } 
 \lim_{\netparam} \Vert (T - T_{\netparam}) f \Vert_{\SO(G_{2})} = 0 \ \ \text{for all} \ \  f\in\SO(G_{1}).$
 \end{enumerate}
\end{proposition}
\begin{proof}
(i). For any $\netparam$ and $i=1,2$ the operator $T_{\netparam}^{(i)}$ belongs to $\mathcal{B}(G_{i},G_{i})$ and thus maps bounded \ws \ convergent nets in $\SOprime(G_{i})$ into norm convergent nets in $\SO(G_{i})$. Since $T\in\Lin(\SO(G_{1}),\SOprime(G_{2}))$ it is clear that then $T_{\netparam} = T_{\netparam}^{(2)}\circ T \circ T_{\netparam}^{(1)}$ also maps bounded \ws \ convergent nets in $\SOprime(G_{1})$ into norm convergent nets in $\SO(G_{2})$. Hence $T_{\netparam}\in\mathcal{B}(G_{1},G_{2})$.\\
(ii). This is a simple estimate:
\begin{align*}
& \lim_{\netparam} \big\vert \big(f^{(2)}, (T- T_{\netparam}^{(2)}\circ T \circ T_{\netparam}^{(1)}) f^{(1)} \big)_{\SO,\SOprime(G_{2})} \big\vert \\
& \le \lim_{\netparam} \big\vert \big(f^{(2)}, (T- T_{\netparam}^{(2)}\circ T) f^{(1)}\big)_{\SO,\SOprime(G_{2})} \big\vert \\
& \quad + \lim_{\netparam} \big\vert \big(f^{(2)},(T_{\netparam}^{(2)}\circ T - T_{\netparam}^{(2)}\circ T \circ T_{\netparam}^{(1)}) f^{(1)} \big)_{\SO,\SOprime(G_{2})} \big\vert \\
& \le \underbrace{\lim_{\netparam} \big\vert \big(f^{(2)}, (\textnormal{Id}_{\SOprime} - T_{\netparam}^{(2)}) (T f^{(1)})\big)_{\SO,\SOprime(G_{2})} \big\vert}_{= 0 \ \ \text{(by Definition \ref{def:reg-app-id}\ref{weak-id})}} \\
& \quad + \lim_{\netparam} \Vert f^{(2)} \Vert_{\SO} \Vert T_{\netparam}^{(2)} \Vert_{\textnormal{op},\SOprime\to\SOprime} \, \Vert T \Vert_{\textnormal{op},\SO\to\SOprime} \, \Vert (\textnormal{Id}_{\SO} - T_{\netparam}^{(1)}) f^{(1)} \Vert_{\SO} \\
& \le \Vert f^{(2)} \Vert_{\SO} \, \Big( \sup_{\netparam} \Vert T_{\netparam}^{(2)} \Vert_{\textnormal{op},\SOprime\to\SOprime} \Big) \, \Vert T \Vert_{\textnormal{op},\SO\to\SOprime} \, \lim_{\netparam} \Vert T_{\netparam} f^{(1)} - f^{(1)} \Vert_{\SO} = 0.
\end{align*}
(iii). This is implied by (ii). \\
(iv). By definition the operators $T_{\netparam}^{(1)}$ and $T_{\netparam}^{(2)}$ have uniformly bounded operator norms as operators on $\SO$ and $\SOprime$. Thus
\begin{align*} & \sup_{\netparam} \Vert T_{\netparam} \Vert_{\textnormal{op},\SO\to\SOprime} = \sup_{\netparam} \Vert T_{\netparam}^{(2)} \circ T \circ T_{\netparam}^{(1)} \Vert_{\textnormal{op},\SO\to\SOprime} \\
& \le \Big( \sup_{\netparam} \Vert T_{\netparam}^{(2)} \Vert_{\textnormal{op},\SOprime\to\SOprime} \Big) \, \Vert T \Vert_{\textnormal{op},\SO\to\SOprime} \, \Big( \sup_{\netparam} \Vert T_{\netparam}^{(1)} \Vert_{\textnormal{op},\SO\to\SO} \Big) < \infty \end{align*}
(vi). In case $T\in\Lin(\SO(G_{1}),\SO(G_{2}))$ we make the following estimate: for all $f\in\SO(G_{1})$
\begin{align*}
& \lim_{\alpha} \Vert (T-T_{\netparam}^{(2)} \circ T \circ T_{\netparam}^{(1)}) f\Vert_{\SO} \\
& \le \underbrace{\lim_{\alpha} \Vert (T-T_{\netparam}^{(2)} \circ T ) f\Vert_{\SO}}_{=0} + \lim_{\alpha} \Vert (T_{\netparam}^{(2)} \circ T -T_{\netparam}^{(2)} \circ T \circ T_{\netparam}^{(1)}) f\Vert_{\SO} \\
& \le \big(\sup_{\alpha} \Vert T_{\netparam}^{(2)} \Vert_{\textnormal{op},\SO\to\SO}  \big) \, \Vert T \Vert_{\textnormal{op},\SO\to\SO} \, \lim_{\alpha} \Vert f - T_{\netparam}^{(1)} f\Vert_{\SO} = 0 
\end{align*}
The proof for (v) is similar.
\end{proof}

We now show one particular application of regularizing approximations of the identity.
\begin{proposition} \label{pr:L2-op-approximation} Consider an operator $S\in \Lin(\Lp{2}(G_{1}),\Lp{2}(G_{2}))$ and $T\in\Lin(\Lp{2}(G_{2}),\Lp{2}(G_{3}))$. Let $(S_{\netparam})$ and $(T_{\netparam})$ be the nets of operators in $\mathcal{B}(G_{1},G_{2})$ and $\mathcal{B}(G_{2},G_{3})$ associated to $S$ and $T$ as in Proposition \ref{pr:reg-of-id}, respectively. In that case the kernel of the operator $T\circ S\in\Lin(\Lp{2}(G_{1}),\Lp{2}(G_{3}))$ is the \ws \ limit of the kernels of the net of operators $(T_{\netparam}\circ S_{\netparam})$, $\kappa(T_{\netparam}\circ S_{\netparam}) \stackrel{\textnormal{w}^{*}}{\longrightarrow} \kappa(T\circ S)$, i.e.,
\[ \lim_{\netparam} \big\vert \big( f^{(3)}, (T\circ S-T_{\alpha}\circ S_{\alpha}) f^{(1)}\big)_{\SO,\SOprime(G_{3})}\big\vert = 0 \ \ \text{for all} \ \ f^{i}\in \SO(G_{i}), \ i=1,3.\]
\end{proposition}
\begin{remark} The usefulness here is that the composition of the operators $S$ and $T$ can we approximated in the \ws \ sense by a composition of operators $S_{\netparam}$ and $T_{\netparam}$ that have kernels in $\SO$. Observe that the composition $T_{\netparam}\circ S_{\netparam}$ is well understood, cf.\ the ``continuous matrix-matrix product'' in Section \ref{sec:linear-albera}.
\end{remark}
\begin{proof}[Proof of Proposition \ref{pr:L2-op-approximation}]
By the estimates  of Proposition \ref{pr:reg-of-id} and the  embedding of $\Lp{2}$  into $\SOprime$ one has: 
\begin{align*}
& \lim_{\netparam} \big\vert \big( f^{(3)}, (T\circ S-T_{\alpha}\circ S_{\alpha}) f^{(1)}\big)_{\SO,\SOprime(G_{3})}\big\vert \\
& \le \lim_{\netparam} \big\vert \big( f^{(3)}, (T\circ S-T_{\alpha}\circ S) f^{(1)}\big)_{\SO,\SOprime(G_{3})}\big\vert\\
& \quad + \lim_{\netparam} \big\vert \big( f^{(3)}, (T_{\alpha}\circ S-T_{\alpha}\circ S_{\alpha}) f^{(1)}\big)_{\SO,\SOprime(G_{3})}\big\vert\\
& \le \Vert f^{(3)} \Vert_{\SO} 
\lim_{\alpha}  \Vert (T-T_{\alpha}) Sf^{(1)} \Vert_{2} \\
& \quad + \Vert f^{(3)} \Vert_{\SO} \big( \sup_{\alpha} \Vert T_{\alpha} \Vert_{\textnormal{op},\SOprime\to\SOprime}\big) \lim_{\alpha} \Vert Sf^{(1)}- S_{\alpha} f^{(1)} \Vert_{2} = 0.
\end{align*}
\end{proof}

Let us apply Proposition \ref{pr:L2-op-approximation} to a concrete example:  
\begin{example} The Fourier transform $\mathcal{F}$ is an operator from $\Lp{2}(G)$ onto $\Lp{2}(\ghat)$. Furthermore its inverse $\mathcal{F}^{-1}$ is an operator from $\Lp{2}(\ghat)$ onto $\Lp{2}(G)$. It is clear that $\mathcal{F}^{-1} \circ \mathcal{F} = \textnormal{Id}_{\Lp{2}(G)}$. It is not difficult to verify that their kernels in $\SOprime$ (guaranteed by the outer kernel theorem) are as follows:
\begin{align*} 
\bullet \ & \kappa(\mathcal{F})\in\SOprime(G\times\ghat) \text{ is induced by the function } G\times\ghat\to\C, (x,\omega) \mapsto \overline{\omega(x)},\\
 & \text{so that } (h, \mathcal{F}f)_{\SO,\SOprime(\ghat)} = \int_{G\times\ghat} f(x) h(\omega) \, \overline{\omega(x)} \, d(x,\omega) \ \ \text{for all} \ \ f\in\SO(G), h\in \SO(\ghat). \\
\bullet \ & \kappa(\mathcal{F}^{-1})\in\SOprime(\ghat\times G) \text{ is induced by the function } \ghat\times G\to\C, (\omega,x) \mapsto \omega(x),\\
& \text{so that } (f, \mathcal{F}^{-1}h)_{\SO,\SOprime(G)} = \int_{\ghat\times G} h(\omega) f(x) \, \omega(x) \, d(x,\omega) \ \ \text{for all} \ \ f\in\SO(G), h\in \SO(\ghat). \\
\bullet \ & \kappa(\textnormal{Id}_{\Lp{2}(G)}) \in \SOprime(G\times G) \text{ is the functional defined by } f_{1} \otimes f_{2} \mapsto \int_{G} f_{1}(x) f_{2}(x) \, dx, \\
& \text{for all $f_{1},f_{2}\in\SO(G)$. This is typically expressed as } \kappa(\textnormal{Id}_{\Lp{2}(G)}) = \delta(y-x). 
\end{align*}

While we describe the distributional kernel for the identity operator as a generalized function of two variables, in the spirit of a Kronecker delta (describing the unit matrix), 
  simply given as $\delta_{Kron}(F) = \int_{G} F(x,x) \, dx$, $F  \in \SO(G\times G)$
 it has become a common understanding to describe the kernel as a continuous collection  of Dirac delta distributions $\delta_y$, or with the usual notation  $\delta(y)$ this becomes just  $\delta(y-x)$.

Let now $(\mathcal{F}_{\netparam})$, $(\mathcal{F}^{-1}_{\netparam})$ be two nets of operators in $\mathcal{B}(G,\ghat)$ and $\mathcal{B}(\ghat,G)$ associated to $\mathcal{F}$ and $\mathcal{F}^{-1}$ as in Proposition \ref{pr:L2-op-approximation}. In that case
\begin{align*} 
& \kappa(\mathcal{F}_{\netparam}) \stackrel{\textnormal{w}^{*}}{\longrightarrow} \kappa(\mathcal{F}), \ \kappa(\mathcal{F}_{\netparam}^{-1}) \stackrel{\textnormal{w}^{*}}{\longrightarrow} \kappa(\mathcal{F}^{-1}),\\
& \kappa(\mathcal{F}^{-1}_{\netparam} \circ \mathcal{F}_{\netparam}) \stackrel{\textnormal{w}^{*}}{\longrightarrow} \kappa(\mathcal{F}^{-1}\circ \mathcal{F}) = \kappa(\textnormal{Id}_{\Lp{2}(G)}).\end{align*}
At the same time Lemma \ref{le:comp-of-operators} tells us that $\kappa(\mathcal{F}^{-1}_{\netparam} \circ \mathcal{F}_{\netparam})$ is the function in $\SO(G\times G)$ given by 
\begin{align*} 
& \kappa(\mathcal{F}^{-1}_{\netparam} \circ \mathcal{F}_{\netparam})(x,y) = \int_{\ghat} \kappa(\mathcal{F}_{\alpha})(x,\omega)\cdot \kappa(\mathcal{F}_{\alpha}^{-1})(\omega,y) \, d\omega. \end{align*}
If ``we take the limit'' of the above integral, then we are lead to the following ``identity'', which is often found in physics and engineering:
\begin{align*} & \int_{\ghat} \kappa(\mathcal{F})(x,\omega) \cdot \kappa(\mathcal{F}^{-1})(\omega,y) \, d\omega = \kappa(\textnormal{Id}_{\Lp{2}(G)}) \\
\Leftrightarrow \quad & \int_{\ghat} \omega(y-x) \, d\omega = \delta(y-x).\end{align*}
Expressed in the familiar setting $G=\ghat=\R$: $\displaystyle\int_{\R} e^{2\pi i \omega(y-x)} \, d\omega = \delta(y-x), x,y\in\R$. 
\end{example}

We now consider examples of regularizing approximations of the identity. 

\begin{example} \textnormal{(Partial sums of Gabor frame operators)} Let $g\in \SO(\R)$ and $a,b>0$ be such that $\{ \pi(\lambda) g\}_{\lambda\in a\Z\times b\Z}$ is a Parseval Gabor frame for $\Lp{2}(\R)$, i.e., 
\[  \Vert f \Vert_{2}^{2} = \sum_{\lambda\in a\Z\times b\Z} \vert \langle f, \pi(\lambda)g \rangle \vert^{2}  \ \ \text{for all} \ \ f\in \Lp{2}(\R).\]
In that case the associated Gabor frame operator
\[ S_{g} : \SO(\R)\to\SO(\R), \ S_{g} f = \sum_{\lambda\in a\Z\times b\Z} \langle f, \pi(\lambda) g\rangle \pi(\lambda)g, \ f\in \SO(\R)\]
is the identity on $\SO(\R)$.
Let $(\Lambda_{N})$, $N\in\N$ be a family of finite subsets of $a\Z\times b\Z$ so that for every point $\lambda\in a\Z\times b\Z$ there exists an $N_{0}\in\N$ such that $N>N_{0}$ implies that $\lambda\in \Lambda_{N}$. For every $N\in\N$ we define the operator
\[ S_{g,N} : \SO(\R)\to\SO(\R), \ S_{g,N} f = \sum_{\lambda\in \Lambda_{N}} \langle f, \pi(\lambda)g\rangle \pi(\lambda) g.\]
It extends to an operator on $\SOprime(\R)$ in the following way:
\[ S_{g,N} : \SOprime(\R)\to\SOprime(\R), \ (f, S_{g,N} \sigma)_{\SO,\SOprime(\R)} = \big( f,  \sum_{\lambda\in \lambda_{N}} (\overline{\pi(\netparam)g},\sigma)_{\SO,\SOprime(\R)} \, \pi(\lambda) g\big)_{\SO,\SOprime(\R)}.\]
The collection of operators $(S_{g,N})_{N\in\N}$ is a regularizing approximation of the identity: It is straight forward to write an explicit formula for the kernel of the operator $S_{g,N}$, namely
\[ \kappa(S_{g,N})(t_{1},t_{2}) = \sum_{\lambda\in\Lambda_{N}} \overline{\pi(\lambda) g(t_{1})} \, \pi(\lambda) g(t_{2}), \ \ t_{1},t_{2}\in\R,\]
such that $(f_{2} , S_{g,N}f_{1})_{\SO,\SOprime(\R)} = (f_{1}\otimes f_{2}, \kappa(S_{g,N}))_{\SO,\SOprime(\R^{2})}$. Hence $\kappa(S_{g,N})\in\SO(\R^{2})$.
Concerning condition \ref{SO-norm-bound} and \ref{SOprime-norm-bound} we need the following two inequalities: 
for any $f\in\SO(\R)$ and $\sigma\in\SOprime(\R)$ there exists a constant $c>0$ such that 
\begin{equation}\label{eq:0702b} \sum_{\lambda\in a\Z\times b\Z} \vert \langle f, \pi(\lambda) g \rangle \vert \le c \, \Vert f \Vert_{\SO} \, \Vert g \Vert_{\SO} \ \ \text{and} \ \ \sup_{\lambda\in a\Z\times b\Z} \vert (\pi(\lambda)g, \sigma) \vert \le c \, \Vert g \Vert_{\SO} \, \Vert \sigma \Vert_{\SOprime}.
\end{equation}
We can then make the following estimates:
\begin{align*}
& \Vert S_{g,N} \Vert_{\textnormal{op},\SO\to\SO} =  \sup_{\substack{f\in\SO(\R)\\ \Vert f \Vert_{\SO}=1}} \Vert S_{g,N} f \Vert_{\SO} \\
& \le \sup_{\substack{f\in\SO(\R)\\ \Vert f \Vert_{\SO}=1}} \big\Vert \sum_{\lambda\in \Lambda_{N}} \langle f, \pi(\lambda)g\rangle \pi(\lambda) g \big\Vert_{\SO} \le \sup_{\substack{f\in\SO(\R)\\ \Vert f \Vert_{\SO}=1}}  \sum_{\lambda\in \Lambda_{N}} \vert \langle f, \pi(\lambda)g\rangle\vert \,  \Vert g \Vert_{\SO} \\
& \le  \sup_{\substack{f\in\SO(\R)\\ \Vert f \Vert_{\SO}=1}}  \sum_{\lambda\in a\Z \times b \Z} \vert \langle f, \pi(\lambda)g\rangle\vert \,  \Vert g \Vert_{\SO} \stackrel{\eqref{eq:0702b}}{\le} c \, \sup_{\substack{f\in\SO(\R)\\ \Vert f \Vert_{\SO}=1}}  \Vert f \Vert_{\SO} \, \Vert g \Vert_{\SO}^{2} = c \, \Vert g \Vert_{\SO}^{2} .
\end{align*}
Hence $\sup_{N} \Vert S_{g,N} \Vert_{\textnormal{op},\SO\to\SO} < \infty$. Similarly, also using \eqref{eq:0702b}, we can show that
\begin{align*} 
& \Vert S_{g,N} \Vert_{\textnormal{op},\SOprime\to\SOprime} \le c \, \Vert g \Vert_{\SO}^{2}.
\end{align*}
Finally, because $S_{g}$ is the identity on $\SO(\R)$ we find that
\begin{align*}
\lim_{N\to\infty} \Vert S_{g,N} f - f \Vert_{\SO} \le \Vert g \Vert_{\SO} \, \lim_{N\to\infty} \sum_{\lambda\in a\Z\times b\Z\backslash\Lambda_{N}} \vert \langle f, \pi(\lambda)g\rangle \vert  = 0
\end{align*}
where the last equality follows from the fact that for any two functions $f,g\in\SO(\R)$ the sequence $\{ \langle f, \pi(\lambda) g\rangle \}_{\lambda\in a\Z\times b\Z}$ is absolutely summable. In a similar way one can show that $(S_{g,N})$ satisfies condition \ref{weak-id} in Definition \ref{def:reg-app-id}.
\end{example}

\begin{example}\textnormal{(Product-convolution operators)} In the sequel  $\mathbf{A}(G)$ is the \emph{Fourier algebra} $\mathbf{A}(G) = \{ f \in C_{0}(G) \, : \, \exists h\in \Lp{1}(\ghat) \ \text{s.t.} \ f = \mathcal{F}_{\ghat}h\}$, here $\mathcal{F}_{\ghat}$ is the Fourier transform from $L^{1}(\ghat)$ into $C_{0}(G)$. The norm in the Fourier algebra is defined by $\Vert f \Vert_{\mathbf{A}} = \Vert h \Vert_{1}$, where $h$ is as before. We now construct regularizing approximations of the identity with the help of product-convolution operators. As described in, e.g., \cite[Proposition 4.18]{ja19}, it is possible to find nets of functions $(h_{\netparam})\in \SO(G)$ and $(g_{\netparam})\in \SO(G)$ such that
\[ \lim_{\netparam} \Vert f * h_{\netparam} - f \Vert_{\SO} = 0 \ \ \text{and} \ \ \lim_{\netparam} \Vert f \cdot g_{\netparam}  - f \Vert_{1} =0 \ \ \forall \ f\in \SO(G),\]
where $\Vert h_{\netparam} \Vert_{1} \le 1$ and $\Vert g_{\netparam} \Vert_{\mathbf{A}(G)} \le 1$ for all $\netparam$. The net of operators
\[ T_{\netparam}: \SOprime(G)\to\SO(G), \ T_{\netparam} \sigma = (\sigma \cdot g_{\netparam})*h_{\netparam}, \ \sigma\in\SOprime(G)\]
is a regularizing approximation of the identity. 
\end{example}

\begin{example}\textnormal{(Localization operators)} Let $(H_{n})$ be a sequence of uniformly bounded functions in $\mathbf{C}_{c}(\R^{2})$ that converges 
uniformly over compact sets 
to the constant function $1$ and take $g$ to be a non-zero function in $\SO(\R)$
with $\|g\|_2 =1$. Then the operators
\[ T_{n} : \SOprime(\R)\to\SO(\R), \ T_{n} \sigma = \int_{\R^{2}} H_{n}(\TFelement) \, (\overline{\pi(\TFelement) g}, \sigma)_{\SO,\SOprime(\R)} \, \pi(\TFelement) g \, d\TFelement\]
form a regularizing approximation of the identity. 

Similar statements can be obtained for Gabor multipliers with respect to
tight Gabor families. 
\end{example}

\subsection{Kernel theorems for modulation spaces}
\label{sec:kernel-for-modspaces}

The inner and outer kernel theorem characterize the operators that are linear and bounded from $\SOprime(G_{1})$ into $\SO(G_{2})$ and from $\SO(G_{1})$ into $\SOprime(G_{2})$, respectively (with some added assumptions in the former case). In between $\SO(G)$ and $\SOprime(G)$, or more precisely, in between the embedding of $\SO(G)$ into $\SOprime(G)$ and $\SOprime(G)$
there is a well-studied family of spaces called the (unweighted) modulation spaces. We refer to \cite{fe03-1,fe06} and the relevant chapters in \cite{gr01} for more on those spaces. 
Meanwhile they are also well presented in the books 
\cite{beok20} and \cite{coro20}. 
For our purpose here we only want to recall the following. 

\begin{definition}   for  $p\in [1,\infty]$, $ g\in\SO(G)\backslash\{0\}$, the modulation space $\mathbf{M}^{p}(G)$ 
is given by
\begin{equation} \label{eq:1403a} \mathbf{M}^{p}(G) = \Big\{ \sigma\in \SOprime(G) \, : \,  \Vert \sigma \Vert_{\mathbf{M}^{p}} := \left ( \int_{G\times \ghat} \big\vert \, ( \pi(\TFelement) g, \sigma )_{\SO,\SOprime} \, \big\vert^{p} \, d\TFelement \right )^{1/p} < \infty \Big\}. \end{equation}
In case $p=\infty$ the definition is modified in the obvious way.
\end{definition}
One can show that different functions $g$ induce equivalent norms. As already mentioned in Section \ref{sec:preliminaries} we have $\Misp(G)\cong \SO(G)$ and $\mathbf{M}^{\infty}(G) = \SOprime(G)$.
For $p\in (1,\infty)$, the modulation space $\mathbf{M}^{p}(G)$ is reflexive and $(\mathbf{M}^{p}(G))' \cong \mathbf{M}^{p'}(G)$, where $1/p + 1/p' = 1$. For any fixed function $g\in \SO(G)\backslash\{0\}$, the action of a generalized function $\sigma\in \mathbf{M}^{p'}(G)$ on a generalized function $f\in \mathbf{M}^{p}(G)$ is given by
\begin{equation} \label{eq:1403b} ( f, \sigma)_{\mathbf{M}^{p},\mathbf{M}^{p'}(G)} = \Vert g \Vert_{2}^{-2} \int_{G\times\ghat} \big( \overline{\pi(\TFelement) g}, f\big)_{\SO,\SOprime(G)} \, \big( \pi(\TFelement)g, \sigma\big)_{\SO,\SOprime(G)} \, d\TFelement.\end{equation}

In light of the inner and outer kernel theorems we may ask: can we characterize the bounded linear operators from $\mathbf{M}^{p}(G)$ into $\mathbf{M}^{q}(G)$ for some $p,q\in[1,\infty]$. It is straight forward to generalize Theorem \ref{th:1608a} to the following sufficient condition for operators in $\Lin(\SO(G_{1}),\SOprime(G_{2}))$ to be operators from $\mathbf{M}^{p'}(G_{1})$ into $\mathbf{M}^{q}(G_{2})$.

\begin{proposition} \label{pr:operator-with-kernel-in-mpq} Fix any two  functions  $g^{(i)}\in \SO(G_{i})\backslash\{0\}$, $i=1,2$ and let $p,q\in[1,\infty]$. If an operator $T\in \Lin(\SO(G_{1}),\SOprime(G_{2}))$ 
satisfies the condition 
\[ \displaystyle \int_{G_{2}\times \ghat_{2}} \Big( \int_{G_{1}\times\ghat_{1}} \big\vert \big(\pi(\TFelement^{(2)}) g^{(2)} , T \pi(\TFelement^{(1)}) g^{(1)}\big)_{\SO,\SOprime(G_{2})} \big\vert^{p} \, d\TFelement^{(1)} \Big)^{q/p} \, d\TFelement^{(2)} < \infty, \]
then $T$ is 
bounded  from $\mathbf{M}^{p'}(G_{1})$ into $\mathbf{M}^{q}(G_{2})$. 
Hence, for $\sigma^{(1)} \in \mathbf{M}^{p'}(G)$ and $\sigma^{(2)}\in \mathbf{M}^{q'}(G)$,
\begin{align*} & \quad \ \Vert g^{(1)} \otimes g^{(2)} \Vert_{2}^{2} \ \big( \sigma^{(2)}, T \sigma^{(1)} \big)_{\mathbf{M}^{q'},\mathbf{M}^{q}} \\
& = \int\limits_{G_{1}\times\ghat_{1} \times G_{2}\times\ghat_{2} }
\mathcal{V}_{g^{(1)}}\sigma^{(1)}(\TFelement^{(1)}) \cdot \mathcal{V}_{g^{(2)}}\sigma^{(2)}(\TFelement^{(2)}) \cdot \big(\pi(\TFelement^{(2)}) g^{(2)} , T \, \pi(\TFelement^{(1)}) g^{(1)}\big)_{\SO,\SOprime(G_{2})}  \, d(\TFelement^{(1)},\TFelement^{(2)}).
\end{align*}
\end{proposition}

In general, the assumption in Proposition \ref{pr:operator-with-kernel-in-mpq} is only \emph{sufficient} for $T$ to be a bounded operator from $\mathbf{M}^{p'}(G_{1})$ to $\mathbf{M}^{q}(G_{2})$. For example, if $p=q=2$, then the identity operator is bounded on $\mathbf{M}^{2}(G) \cong \Lp{2}(G)$, but its kernel is not in $\Lp{2}(G\times G)$.


Recently, in \cite{bagr17,bagrsp19} and \cite{coni17} it has been shown that for 
\[ \text{(1) \ $p=\infty$  and $q\in [1,\infty]$ \ \ \textnormal{and} \ \ \ (2) \ $p\in [1,\infty]$ and $q=\infty$
 }\]
 it is possible to give (relatively abstract) necessary and sufficient conditions for operators to be continuous from $\mathbf{M}^{p'}(G_{1})$ to $\mathbf{M}^{q}(G_{2})$ in terms of the kernels belonging to certain modulation spaces. Such results 
confirm the usefulness of coorbit spaces, here specifically of modulation spaces.

\section{Proof of the inner kernel theorem} \label{sec:proof}

Here we give the proof for the inner kernel theorem, Theorem \ref{th:new-inner-kernel-theorem}. It is useful to introduce the space $\tilde{\mathcal{B}}(G_{1},G_{2})$:

\begin{definition} \label{def:B-tilde} Let $G_{1}$ and $G_{2}$ be locally compact abelian Hausdorff groups. We then define
\begin{enumerate}
\item[]  $ \mathcal{\tilde{B}}(G_{1},G_{2}) =\{ T \in \Lin(\SOprime(G_{1}),\iota(\SO(G_{2}))) \, : \, T$ maps every bounded \ws convergent net in $\SOprime(G_{1})$ into a norm convergent net in $\iota(\SO(G_{2}))\subseteq\SOdoubleprime(G_{2}) \}$.
\end{enumerate}
\end{definition} 

\noindent 
The identification of $\SO(G)$ with $\iota(\SO(G))$ (see Section \ref{sec:S0}) implies that $\mathcal{\tilde{B}}(G_{1},G_{2})\cong \mathcal{B}(G_{1},G_{2})$.

\begin{proof}[Proof of Theorem \ref{th:new-inner-kernel-theorem}]
We will show that the three Banach spaces $\SO(G_{1}\times G_{2})$, $\mathcal{A}$ and $\tilde{\mathcal{B}}(G_{1},G_{2})$ are isomorphic. By the isomorphism between $\mathcal{\tilde{B}}(G_{1},G_{2})\cong \mathcal{B}(G_{1},G_{2})$ and the fact that $\SO(G_{1}\times G_{2})\cong \SO(G_{2}\times G_{1})$ the inner kernel theorem follows. 
In order to prove the desired identifications, we consider the following two operators.
\begin{align*} 
c & : \SO(G_{1}\times G_{2}) \to \mathcal{A}, \ c(K) = \Big[ (\sigma^{(1)},\sigma^{(2)})\mapsto (K, \sigma^{(1)}\otimes \sigma^{(2)})_{\SO,\SOprime(G_{1}\times G_{2})} \Big], \\
d  & : \mathcal{A}\to \tilde{\mathcal{B}}(G_{1},G_{2}), \ d(A) = \Big[ \sigma^{(1)} \mapsto \big[  \sigma^{(2)} \mapsto A(\sigma^{(1)},\sigma^{(2)} ) \ \big]\ \Big],
\end{align*}
where $K\in \SO(G_{1}\times G_{2})$, $A\in \mathcal{A}$ and $\sigma^{(i)}\in \SOprime(G_{i})$, $i=1,2$. 
In Lemma \ref{le:c-operator} and Lemma \ref{le:d-operator} we will show that both these operators are well-defied, linear and bounded. 

Furthermore, let $\SOprime(G_{1})\otimes \SOprime(G_{2})$ be the tensor product of $\SOprime(G_{1})$ and $\SOprime(G_{2})$, that is, the linear span of elementary tensors, 
\[ \SOprime(G_{1})\otimes \SOprime(G_{2}) = \{ \sigma \in \SOprime(G_{1}\times G_{2}) \, : \, \sigma = \textstyle\sum_{j=1}^{N} \sigma^{(1)}_{j}\otimes \sigma^{(2)}_{j}, \ N\in \N \,\} .\]
Then, for a given $T\in \tilde{\mathcal{B}}(G_{1},G_{2})$, we define the operator
\[ e(T) : \SOprime(G_{1})\otimes\SOprime(G_{2}) \to \C, \ e(T) \Big( \sum_{j=1}^{N} \sigma^{(1)}_{j} \otimes \sigma^{(2)}_{j}\Big) = \sum_{j=1}^{N} T(\sigma^{(1)}_{j})(\sigma^{(2)}_{j}). \]
So far, it is not clear whether the value of $e(T)(\sigma)$, $\sigma\in \SOprime(G_{1})\otimes \SOprime(G_{2})$   depends on the particular representation $\sum_{j=1}^{N} \sigma^{(1)}_{j} \otimes \sigma^{(2)}_{j}$ of $\sigma$. We will show in a moment that this is not the case.

In Lemma \ref{le:e-operator} we show that $e(T)$ is continuous with respect to the \ws \ topology induced by functions in $\SO(G_{1}\times G_{2})$. Because $\SOprime(G_{1})\otimes \SOprime(G_{2})$ is \ws \ dense in $\SOprime(G_{1}\times G_{2})$ (this is the case because the generalized functions induced by $\SO(G_{1})\otimes \SO(G_{2})$ are \ws \ dense in $\SOprime(G_{1}\times G_{2})$ and they are a subspace of $\SOprime(G_{1})\otimes \SOprime(G_{2})$), there is a unique \ws \ continuous extension of $e(T)$, which we also call $e(T)$, to a functional from $\SOprime(G_{1}\times G_{2})$ to $\C$. We can therefore define the operator 
\[ e: \tilde{\mathcal{B}} \to \iota(\SO(G_{1}\times G_{2}))\subseteq \SOdoubleprime(G_{1}\times G_{2}),\]
which, to every $T\in \tilde{\mathcal{B}}$, assigns the operator $e(T)$ from above. Since $\iota(\SO(G_{1}\times G_{2}))\cong \SO(G_{1}\times G_{2})$ we can consider $e$ as an operator from $\tilde{\mathcal{B}}$ into $\SO(G_{1}\times G_{2})$. 

Now, given $K\in \SO(G_{1}\times G_{2})$, $A\in \mathcal{A}$ and $T\in \tilde{\mathcal{B}}(G_{1},G_{2})$ one can, simply by the definitions of the three operators $c$, $d$ and $e$, show that
\[ e \circ d \circ c(K) =  K, \ \ c\circ e\circ d (A) = A, \ \ d\circ c \circ e(T) = T.\]
This implies that $c$, $d$ and $e$ are injective, surjective, and hence invertible. We conclude that $e$ is the (unique) inverse operator of $d\circ c$, thus $e(T)(\sigma)$ for $\sigma\in \SOprime(G_{1})\otimes \SOprime(G_{2})$ can \emph{not} depend on a particular representation of $\sigma$ as discussed earlier in the proof. Because $\SO(G_{1}\times G_{2})$ is a Banach space, it follows that also the normed vector spaces $\mathcal{A}$ and $\tilde{\mathcal{B}}(G_{1},G_{2})$ are Banach spaces. To complete the proof it remains only to prove Lemma \ref{le:c-operator}, \ref{le:d-operator}, and \ref{le:e-operator}.
\end{proof}

In order to verify weak$^{*}$ continuity of functionals the following result is essential to us.
\begin{lemma}[{\cite[Corollary 2.7.9]{me98-1}}] \label{le:weak-star-continuity-of-functionals}
Let $\Xsp$ be a Banach space and $\Xsp'$ its continuous dual space. For a functional $\varphi:\Xsp'\to \C$ the following statements are equivalent:
\begin{enumerate}
\item[(i)] $\varphi$ is \ws \ continuous, i.e., if $(x'_{\netparam})$ is a \ws \ convergent net in $\Xsp'$ with limit $x'_{0}$, then for all $\epsilon>0$ there exists a $\netparam_{0}$ such that for all $\netparam > \netparam_{0}$ one has $ \vert \varphi(x'_{\netparam}-x'_{0}) \vert < \epsilon.$ 
\item[(ii)] $\varphi$ is continuous with respect to the bounded \ws \ topology, i.e., if $(x'_{\netparam})$ is a (in $\Xsp'$ norm) bounded \ws \ convergent net in $\Xsp'$ with limit $x'_{0}$, then for all $\epsilon>0$ there exists a $\netparam_{0}$ such that one has $ \vert \varphi(x'_{\netparam}-x'_{0}) \vert < \epsilon \, $ for all $\netparam > \netparam_{0}$.  
\end{enumerate}
\end{lemma} 

\begin{lemma} \label{le:c-operator}
The operator 
\[ c : \SO(G_{1}\times G_{2}) \to \mathcal{A}, \ c(K) = \left[ (\sigma^{(1)},\sigma^{(2)})\mapsto (K, \sigma^{(1)}\otimes \sigma^{(2)})_{\SO,\SOprime(G_{1}\times G_{2})} \right ].  \]
is well-defined, linear and bounded.
\end{lemma}
\begin{proof}
Let a function $K\in \SO(G_{1}\times G_{2})$ be given. Then for some constant $ a > 0$
\begin{align} \vert \, c(K)(\sigma^{(1)},\sigma^{(2)}) \vert & = \vert ( K, \sigma^{(1)}\otimes \sigma^{(2)}) \vert \le \Vert K \Vert_{\SO}  \Vert \sigma^{(1)}\otimes \sigma^{(2)} \Vert_{\SOprime} \nonumber \\
& \stackrel{\eqref{eq:soprime-tensor-norm}}{\le} a \, \Vert K \Vert_{\SO} \, \Vert \sigma^{(1)} \Vert_{\SOprime} \, \Vert \sigma^{(2)} \Vert_{\SOprime}, \label{eq:040817a}  \end{align}
Hence $c(K)(\sigma^{(1)},\sigma^{(2)})$ is well-defined. The bilinearity of $c(K)$ is clear. Also, 
\begin{equation} \label{eq:040817b} \sup_{\substack{ \Vert \sigma^{(i)} \Vert_{\SOprime(G_{i})} = 1, \, \, i=1,2 }} \vert c(K)(\sigma^{(1)},\sigma^{(2)}) \vert \le a \, \Vert K \Vert_{\SO}.\end{equation}
This shows that $c(K)$ is an element in $\Bil(\SOprime(G_{1})\times \SOprime(G_{2}), \C)$. Let us show that $c(K)\in \mathcal{A}$, i.e., $c(K)$ is \ws \ continuous in each variable. In order to show this, let us first consider a function $K\in \SO(G_{1})\otimes \SO(G_{2})\subseteq \SO(G_{1}\times G_{2})$, that is, a function of the form
\[ K = \sum_{j=1}^{N} f^{(1)}_{j} \otimes f^{(2)}_{j}, \ (f^{(i)}_{j})_{j=1}^{N} \ \text{in} \ \SO(G_{i}), \ i=1,2, \ N\in\N.\]
If $(\sigma^{(1)}_{\netparam})$ is a bounded \ws \ convergent net in $\SOprime(G_{1})$ with limit $\sigma_{0}^{(1)}$, and  $\sigma^{(2)}\in \SOprime(G_{2})$, then 
\begin{align*}
& \quad \ \lim_{\netparam} \vert c(K)(\sigma^{(1)}_{\netparam} - \sigma^{(1)}_{0}, \sigma^{(2)})\vert
 = \lim_{\netparam} \big\vert \big(K,(\sigma^{(1)}_{\netparam}-\sigma^{(1)}_{0}) \otimes  \sigma^{(2)}\big) \big\vert \\
& = \lim_{\netparam} \vert \sum_{j=1}^{N} (f_{j}^{(1)},\sigma_{\netparam}^{(1)}-\sigma_{0}^{(1)}) \, (f_{j}^{(2)}, \sigma^{(2)}) \, \vert \\
& \le a \, \max_{j} \Vert f_{j}^{(2)} \Vert_{\SO} \, \Vert \sigma^{(2)}\Vert_{\SOprime} \, \sum_{j=1}^{N} \lim_{\netparam} \vert (f_{j}^{(1)}, \sigma_{\netparam}^{(1)}-\sigma_{0}^{(1)}) \vert = 0.
\end{align*}
By Lemma \ref{le:weak-star-continuity-of-functionals} the operator $c(K)$ is \ws \ continuous in the first coordinate. The continuity in the second coordinate is proven in the same fashion. Let now $K$ be any function in $\SO(G_{1}\times G_{2})$. Then, given any $\epsilon>0$, we can find a function $\tilde{K}\in \SO(G_{1})\otimes \SO(G_{2})$ such that 
\[ \Vert K - \tilde{K} \Vert_{\SO} \, \cdot \sup_{\netparam,\{0\}} \Vert \sigma^{(1)}_{(\cdot)} \Vert_{\SOprime} \, \Vert \sigma^{(2)} \Vert_{\SOprime}  < \frac{\epsilon}{4} .\]
With this $\tilde{K}$ fixed, there is, as we just showed, an index $\netparam_{0}$ such that for all $\netparam>\netparam_{0}$ 
\[ \vert c(\tilde{K})(\sigma^{(1)}_{\netparam}-\sigma^{(1)}_{0}, \sigma^{(2)}) \vert < \epsilon/2.\]
Hence, for $\netparam>\netparam_{0}$ we have that
\begin{align*}
& \quad \ \vert \, c(K)(\sigma^{(1)}_{\netparam}-\sigma^{(1)}_{0}, \sigma^{(2)}) \vert \\
& = \vert \, c(K-\tilde{K}+\tilde{K})(\sigma^{(1)}_{\netparam}-\sigma^{(1)}_{0}, \sigma^{(2)}) \vert \\
& \le \vert \, c(K-\tilde{K})(\sigma^{(1)}_{\netparam}-\sigma^{(1)}_{0}, \sigma^{(2)}) \vert + \vert \,  c(\tilde{K})(\sigma^{(1)}_{\netparam}-\sigma^{(1)}_{0}, \sigma^{(2)})\vert\\
& < 2 \, \Vert K - \tilde{K} \Vert_{\SO} \, \sup_{\netparam,\{0\}} \Vert \sigma^{(1)}_{(\cdot)} \Vert_{\SOprime} \, \Vert \sigma^{(2)} \Vert_{\SOprime} + \epsilon/2\\
& < \epsilon/2 + \epsilon/2 = \epsilon.
\end{align*}
We have thus shown that $c(K)$ is \ws \ continuous in the first coordinate for any $K\in\SO(G_{1}\times G_{2})$. The continuity in the second coordinate is proven in the same way. 
Consequently 
$c$ is a mapping from $\SO(G_{1}\times G_{2})$ into $\mathcal{A}$. The linearity of $c$ is clear.  
Finally, the boundedness of 
$c$ follows from the   inequalities concerning $c(K)$ above, namely,
\begin{align*}
\sup_{\substack{K\in \SO(G_{1}\times G_{2}) \\ \Vert K \Vert_{\SO} = 1}} \Vert c(K) \Vert_{\Bil(\SOprime\times\SOprime,\C)} \le a,
\end{align*} 
where $a$ is the same constant as in \eqref{eq:040817a} and \eqref{eq:040817b}. Hence the operator $c$ is well-defined, linear and bounded.
\end{proof}
\begin{lemma} \label{le:d-operator} The operator
\[ d  : \mathcal{A}\to \tilde{\mathcal{B}}(G_{1},G_{2}), \ d(A) = \Big[ \sigma^{(1)} \mapsto \big[  \sigma^{(2)} \mapsto A(\sigma^{(1)},\sigma^{(2)} ) \ \big] \, \Big], \ \sigma^{(i)}\in \SOprime(G_{i}), \ i=1,2,\]
is well-defined, linear and bounded.
\end{lemma}

\begin{proof} Let $A$ be an operator in $\mathcal{A}$. Let us show that $d(A)$ is an operator in $\tilde{\mathcal{B}}(G_{1},G_{2})$. That is, we need to show that $d(A)\in \Lin(\SOprime(G_{1}), \iota(\SOprime(G_{2})))$ and that $d(A)$ maps bounded \ws \ convergent nets in $\SOprime(G_{1})$ into norm convergent nets in $\SOdoubleprime(G_{2})$. 
Since $A\in\mathcal{A}$ it is clear that for all $\sigma^{(1)}\in\SOprime(G_{1})$ and $\sigma^{(2)}\in\SOprime(G_{2})$ we have the estimate
\begin{equation} \label{eq:0112c} \vert d(A)(\sigma^{(1)})(\sigma^{(2)}) \vert = \vert A(\sigma^{(1)},\sigma^{(2)}) \vert \le \Vert A \Vert_{\textnormal{op}} \Vert \sigma^{(1)} \Vert_{\SOprime} \, \Vert \sigma^{(2)} \Vert_{\SOprime} < \infty. \end{equation}
Hence the functional
\[ d(A)(\sigma_{1}): \SOprime(G_{2}) \to \C, \ d(A)(\sigma^{(1)})(\sigma^{(2)}) = A(\sigma^{(1)},\sigma^{(2)})\] 
is well-defined. The bilinearity of $A$ implies that $d(A)(\sigma^{(1)})$ is linear. In order to show that the functional is also bounded we use the estimate from \eqref{eq:0112c}. This yields
\begin{equation} \label{eq:0112d} \sup_{\substack{\sigma^{(2)}\in\SOprime(G_{2}) \\ \Vert \sigma^{(2)} \Vert = 1}} \vert d(A)(\sigma^{(1)})(\sigma^{(2)}) \vert \stackrel{\eqref{eq:0112c}}{\le} \sup_{\substack{\sigma^{(2)}\in\SOprime(G_{2}) \\ \Vert \sigma^{(2)} \Vert = 1}} \Vert A \Vert_{\textnormal{op}} \, \Vert \sigma^{(1)} \Vert_{\SOprime} \, \Vert \sigma^{(2)} \Vert_{\SOprime} =   \Vert A \Vert_{\textnormal{op}} \, \Vert \sigma^{(1)} \Vert_{\SOprime}<\infty.\end{equation}
Hence $d(A)(\sigma^{(1)})$ is also bounded. The \ws continuity of this functional is also easy to show: if $(\sigma^{(2)}_{\netparam})$ is a \ws \ convergent net in $\SOprime(G_{2})$ with limit $\sigma^{(2)}_{0}\in \SOprime(G_{2})$, then, since $A$ is \ws \ continuous in the second coordinate, 
\[ \lim_{\netparam} \vert d(A)(\sigma^{(1)})(\sigma^{(2)}_{\netparam}-\sigma^{(2)}_{0}) \vert = \lim_{\netparam} \vert A(\sigma^{1},\sigma^{(2)}_{\netparam}-\sigma^{(2)}_{0}) \vert = 0.\]
Thus $d(A)(\sigma^{(1)})\in\iota(\SO(G_{2}))$. Let us verify that $d(A)$ is a bounded operator from $\SOprime(G_{1})$ into $\iota(\SO(G_{2}))\subseteq \SOdoubleprime(G_{2})$.
\begin{equation} \label{eq:0112e} \sup_{\substack{ \Vert \sigma^{(1)}
\Vert_{\SOprime(G_{1})} \leq 1 }} \Vert d(A)(\sigma^{(1)}) \Vert_{\SOdoubleprime} \stackrel{\eqref{eq:0112d}}{\le} \sup_{\substack{  \Vert \sigma^{(1)} \Vert_{\SOprime(G_{1})} \leq 1 }} \Vert A \Vert_{\textnormal{op}} \, \Vert \sigma^{(1)} \Vert_{\SOprime} = \Vert A \Vert_{\textnormal{op}}. 
\end{equation} 
We have thus shown that $d(A)\in \Lin(\SOprime(G_{1}),\iota(\SO(G_{2})))$. It is left to show that $d(A)$ maps bounded \ws \ convergent nets in $\SOprime(G_{1})$ into norm convergent nets in $\iota(\SO(G_{2}))\subseteq \SOdoubleprime(G_{2})$. Given a bounded \ws \ convergent net $(\sigma^{(1)}_{\netparam})$   in $\SOprime(G_{1})$ with limit $\sigma^{(1)}_{0}$ one has: 
\begin{align*} \lim_{\netparam} \Vert d(A)(\sigma^{(1)}_{\netparam}-\sigma^{(1)}_{0}) \Vert_{\SOdoubleprime} & = \lim_{\netparam} \sup_{\substack{ \Vert \sigma^{(2)} \Vert_{\SOprime(G_{2})} \le 1  }}  \vert d(A)(\sigma^{(1)}_{\netparam}-\sigma^{(1)}_{0})(\sigma^{(2)}) \vert \\ & = \lim_{\netparam} \sup_{\substack{ \Vert \sigma^{(2)} \Vert_{\SOprime(G_{2})} \le 1  }} \vert A(\sigma^{(1)}_{\netparam}-\sigma^{(1)}_{0},\sigma^{(2)}) \vert. \end{align*}
We need to show that the limit is equal to zero. Note that $A$ is \ws \ continuous in the first and second entry. By the Banach-Alaoglu Theorem (\cite[Theorem 2.6.18]{me98-1}) the unit ball of $\SOprime(G_{2})$ is compact in the \ws \ topology. Continuous mappings on compact sets are uniformly  continuous,  therefore we  conclude that 
\[ \lim_{\netparam} \Vert d(A)(\sigma^{(1)}_{\netparam}-\sigma^{(1)}_{0}) \Vert_{\SOdoubleprime} = \lim_{\netparam} \sup_{\substack{    \Vert \sigma^{(2)} \Vert_{\SOprime(G_{2}) } \le 1  }} \vert A(\sigma^{(1)}_{\netparam}-\sigma^{(1)}_{0},\sigma^{(2)}) \vert = 0.\]
Hence $d(A)(\sigma^{(1)}_{\netparam})$ is a $\SOdoubleprime(G_{2})$-norm convergent net with limit $d(A)(\sigma^{(1)}_{0})$. Thus $d(A)\in \tilde{\mathcal{B}}$ and hence $d$ is a well-defined operator and clearly linear. It is also bounded: 
\[ \Vert d \Vert_{\textnormal{op}} = \sup_{ \Vert A \Vert = 1} \Vert d(A) \Vert_{\textnormal{op},\SOprime\to\SOdoubleprime} \stackrel{\eqref{eq:0112e}}{\le} 1.  \]
\end{proof}  

\begin{lemma} \label{le:e-operator} For every $T\in \tilde{\mathcal{B}}(G_{1},G_{2})$, the operator given by 
\[ e(T) : \SOprime(G_{1})\otimes\SOprime(G_{2}) \to \C, \ e(T) \Big( \sum_{j=1}^{N} \sigma^{(1)}_{j} \otimes \sigma^{(2)}_{j}\Big) = \sum_{j=1}^{N} T(\sigma^{(1)}_{j})(\sigma^{(2)}_{j}). \]
is 
linear and continuous with respect to the \ws \ topology induced by 
$\SO(G_{1}\times G_{2})$.
\end{lemma}
\begin{proof}
Let us first show that $e(T)$ is a well-defined and linear operator on $\SOprime(G_{1})\otimes\SOprime(G_{2})$. Indeed, we find that for all finite sequences $(\sigma^{(i)}_{j})_{j=1}^{N}$ in $\SOprime(G_{i})$, $i=1,2$, $N\in\N$, 
\[\big \vert e(T)\left({\textstyle\sum\limits_{j=1}\limits^{N}} \sigma^{(1)}_{j} \otimes \sigma^{(2)}_{j} \right) \big \vert = \big \vert \sum_{j=1}^{N} T(\sigma^{(1)}_{j})(\sigma^{(2)}_{j}) \big \vert \le \Vert T \Vert_{\textnormal{op},\SOprime\to\SOdoubleprime} \, \sum_{j=1}^{N} \Vert \sigma^{(1)}_{j} \Vert_{\SOprime} \, \Vert \sigma^{(2)}_{j} \Vert_{\SOprime} < \infty. \]
For $e(T)$ to be well-defined we should verify that the value of $e(T)(\sigma)$, $\sigma\in \SOprime(G_{1})\otimes \SOprime(G_{2})$ is independent of its particular representation $\sum_{j=1}^{N}\sigma^{(1)}_{j} \otimes \sigma^{(2)}_{j}$. This issue is resolved in the proof of Theorem \ref{th:new-inner-kernel-theorem}. The linearity of the operator $e(T)$ follows immediately from its definition. 

Let us now show that $e(T)$ is \ws \ continuous. 
That is, we wish to show that if a bounded net of elementary tensors,  $(\sigma^{(1)}_{\netparam}\otimes \sigma^{(2)}_{\netparam})$ is \ws \ convergent towards $\sigma^{(1)}_{0} \otimes \sigma^{(2)}_{0}$, then
\begin{equation} \label{eq:3011a} \lim_{\netparam} e(T)(\sigma^{(1)}_{\netparam} \otimes \sigma^{(2)}_{\netparam}) = e(T)(\sigma^{(1)}_{0}\otimes \sigma^{(2)}_{0}).\end{equation}

Since $e(T)$ is linear, it is enough to verify its \ws \ continuity at $0$. We may write the zero element in $\SOprime(G_{1})\otimes \SOprime(G_{2})$ as $0 = \sigma^{(1)}_{0} \otimes \sigma^{(2)}_{0}$, where $\sigma^{(1)}_{0} = 0$ and $\sigma^{(2)}_{0}$ is some non-zero element in $\SOprime(G_{2})$ with $\Vert \sigma_{0}^{(2)}\Vert_{\SOprime} = 2$. Assume now that $(\sigma^{(1)}_{\netparam}\otimes \sigma^{(2)}_{\netparam})\stackrel{\textnormal{w}^{*}}{\longrightarrow} 0 = (0\otimes \sigma^{(2)}_{0})$. Furthermore,  we may assume without loss of generality  
that $\sup_{\netparam} \Vert \sigma^{(1)}_{\netparam} \Vert_{\SOprime(G_{1})} < \infty$ and 
$\Vert \sigma^{(2)}_{\netparam} \Vert_{\SOprime} = 1$ for all $\netparam$ (in order to achieve this normalization use that $\sigma^{(1)}\otimes \sigma^{(2)} = \alpha \sigma^{(1)} \otimes \alpha^{-1}\sigma^{(2)}$ for all $\alpha\in\C\backslash\{0\}$. If $\sigma^{(2)}_{\netparam}=0$, then use that $\sigma^{(1)}\otimes 0 = 0 \otimes \sigma^{(2)} = 0$ and then normalize appropriately). 

Assume for a moment that $(\sigma^{(1)}_{\netparam}) \stackrel{\textnormal{w}^{*}}{\longarrownot \longrightarrow} 0$ and that $(\sigma^{(2)}_{\netparam}) \stackrel{\textnormal{w}^{*}}{\longarrownot \longrightarrow} \sigma^{(2)}_{0}$. Then, for $i=1,2$ there exist a function $h^{(i)}\in \SO(G_{i})$ and an $\epsilon^{(i)}>0$ such that, for all index $\netparam^{(i)}_{0}$ we have that $\netparam^{(i)}>\netparam^{(i)}_{0}$ and $\vert ( h^{(i)},\sigma^{(i)}_{\netparam^{(i)}}-\sigma^{(i)}_{0})_{\SO,\SOprime(G_{i})}\vert \ge \epsilon^{(i)}$. This allows us, for sufficiently large $\netparam$, to achieve the inequality
\begin{equation} \label{eq:3011b} {\epsilon^{(1)}}{\epsilon^{(2)}} \le \vert (h^{(1)},\sigma^{(1)}_{\netparam})_{\SO,\SOprime(G_{1})} (h^{(2)},\sigma^{(2)}_{\netparam}-\sigma^{(2)}_{0})_{\SO,\SOprime(G_{2})}\vert. \end{equation}
On the other hand, because by assumption $(\sigma^{(1)}_{\netparam}\otimes \sigma^{(2)}_{\netparam})\stackrel{\textnormal{w}^{*}}{\longrightarrow} (0\otimes \sigma^{(2)}_{0}) = 0$ we can ensure that, for sufficiently high values of $\netparam$,
\[ \vert (h^{(1)},\sigma^{(1)}_{\netparam})_{\SO,\SOprime(G_{1})} \, (h^{(2)},\sigma^{(2)}_{\netparam}-\sigma^{(2)}_{0})_{\SO,\SOprime(G_{2})}\vert < {\epsilon}^{(1)}{\epsilon}^{(2)}.\]
This is a contradiction to \eqref{eq:3011b} and therefore the assumption that $(\sigma^{(1)}_{\netparam}) \stackrel{\textnormal{w}^{*}}{\longarrownot \longrightarrow} 0$ and that $(\sigma^{(2)}_{\netparam}) \stackrel{\textnormal{w}^{*}}{\longarrownot\longrightarrow} \sigma^{(2)}_{0}$ is wrong. We must therefore be in either of the following three situations:
\begin{enumerate}
\item[(i)] $(\sigma^{(1)}_{\netparam}) \stackrel{\textnormal{w}^{*}}{\longrightarrow} 0$ and $(\sigma^{(2)}_{\netparam}) \stackrel{\textnormal{w}^{*}}{\longarrownot\longrightarrow} \sigma^{(2)}_{0}$
\item[(ii)] $(\sigma^{(1)}_{\netparam}) \stackrel{\textnormal{w}^{*}}{\longrightarrow} 0$ and $(\sigma^{(2)}_{\netparam})  \stackrel{\textnormal{w}^{*}}{\longrightarrow} \sigma^{(2)}_{0}$
\item[(iii)] $(\sigma^{(1)}_{\netparam}) \stackrel{\textnormal{w}^{*}}{\longarrownot\longrightarrow} 0$ and $(\sigma^{(2)}_{\netparam}) \stackrel{\textnormal{w}^{*}}{\longrightarrow} \sigma^{(2)}_{0}$
\end{enumerate}
Assume for a moment that $(\sigma^{(2)}_{\netparam}) \stackrel{\textnormal{w}^{*}}{\longrightarrow} \sigma_{2}^{(0)}$. It follows from \cite[Theorem 2.6.14]{me98-1} that this implies  
\[ \Vert \sigma^{(2)}_{0} \Vert_{\SOprime} \le \liminf_{\netparam} \Vert \sigma^{(2)}_{\netparam} \Vert_{\SOprime}.\]
However, with our choice of normalization we find that
\[ 2 = \Vert \sigma^{(2)}_{0} \Vert_{\SOprime} \le \liminf_{\netparam} \Vert \sigma^{(2)}_{\netparam} \Vert_{\SOprime} = 1,\]
which, clearly, can not be the case. We must therefore be in situation (i). We thus have that $(\sigma^{(1)}_{\netparam}) \stackrel{\textnormal{w}^{*}}{\longrightarrow} \sigma^{(1)}_{0} = 0$. Note that $T$ maps bounded \ws \ convergent nets in $\SOprime(G_{1})$ into norm convergent nets in $\SOdoubleprime(G_{2})$. Thus $\lim_{\netparam} \Vert T \sigma^{(1)}_{\netparam} \Vert_{\SOdoubleprime} = 0$. We therefore find that
\begin{align*} & \quad \, \lim_{\netparam} \vert e(T)(\sigma^{(1)}_{\netparam}\otimes\sigma^{(2)}_{\netparam}) - e(T)(\sigma^{(1)}_{0}\otimes \sigma^{(2)}_{0})\vert \\
& = \lim_{\netparam} \vert T(\sigma^{(1)}_{\netparam})(\sigma^{(2)}_{\netparam}) - T(0)(\sigma^{(2)}_{0}) \vert  = \lim_{\netparam} \vert T(\sigma^{(1)}_{\netparam})(\sigma^{(2)}_{\netparam}) \vert \\
& \le \lim_{\netparam} \Vert T(\sigma^{(1)}_{\netparam}) \Vert_{\SOdoubleprime} \, \Vert \sigma^{(2)}_{\netparam} \Vert_{\SOprime} \\
& \le  \big( \sup_{\netparam} \Vert \sigma^{(2)}_{\netparam} \Vert_{\SOprime} \big) \, \lim_{\netparam} \Vert T(\sigma^{(1)}_{\netparam}) \Vert_{\SOdoubleprime} = 0.
 \end{align*}
We have thus verified the continuity of $e(T)$ for elementary tensors with respect to the \ws \ topology induced by functions in $\SO(G_{1}\times G_{2})$. This continuity is preserved by finite linear combinations and as a consequence $e(T)$ is continuous from $\SOprime(G_{1})\otimes \SOprime(G_{2})$ into $\C$.
\end{proof}

\begin{remark}
It is worthwhile to note that in the proof of the inner kernel theorem we only used that the linear span of the elementary tensors in $\SO$, i.e., $\SO(G_{1})\otimes \SO(G_{2})$, is norm-dense in $\SO(G_{1}\times G_{2})$ and that the linear span of the elementary tensors in $\SOprime$, i.e., $\SOprime(G_{1})\otimes \SOprime(G_{2})$ is \ws-dense in $\SOprime(G_{1}\times G_{2})$.
It is therefore possible to formulate the inner kernel theorem in a more general setting with the necessary assumptions, e.g., for coorbit spaces (cf.\ \cite{bagrsp19}) or general Banach spaces. We leave this for elsewhere. See also Remark \ref{rem:2601b}.
 
\end{remark}

\subsection*{Acknowledgments}
We are grateful to the anonymous referee. The comments helped to widen the scope of the presentation in the manuscript. 
The first version of the manuscript was finished (spring 2017)  while the first author was a guest professor at the Technical University of Munich  
(Chair of Theoretical Information Technology: Holger Boche). 
The work of M.S.J.\ was carried out during the tenure of the ERCIM 'Alain Bensoussan` Fellowship Programme at NTNU and finished with support by Deutsches Elektronen-Synchrotron DESY and HamburgX grant LFF-HHX-03 to the Center for Data and Computing in Natural Sciences (CDCS) from the Hamburg Ministry of Science, Research, Equalities and Districts.

\bibliographystyle{abbrv}


\end{document}